\documentclass[10pt]{amsart} 
\usepackage{amsmath,amssymb,amsthm,amsfonts,amsxtra,mathtools}
\usepackage{enumerate,verbatim,mathrsfs,comment}
\usepackage[all,2cell,ps]{xy}
\usepackage[pagebackref]{hyperref}
\usepackage{graphicx}
\usepackage{verbatim}
\newtheorem{headtheorem}{Theorem}

\newtheorem{headexample}{Example}

\DeclareMathOperator{\ann}{ann}

\DeclareMathOperator{\Coker}{Coker}
\DeclareMathOperator{\Cone}{Cone}

\DeclareMathOperator{\depth}{depth}
\DeclareMathOperator{\eend}{end}
\DeclareMathOperator{\Ext}{Ext}

\DeclareMathOperator{\Hom}{Hom}
\DeclareMathOperator{\Char}{char}
\DeclareMathOperator{\Image}{Image}
\DeclareMathOperator{\indeg}{indeg}

\DeclareMathOperator{\Ker}{Ker}

\DeclareMathOperator{\Proj}{Proj}

\DeclareMathOperator{\reg}{reg}

\DeclareMathOperator{\Specmax}{Specmax}
\DeclareMathOperator{\Supp}{Supp}

\DeclareMathOperator{\Tor}{Tor}

\renewcommand{\ge}{\geqslant}
\renewcommand{\le}{\leqslant}

\theoremstyle{plain}
\newtheorem{theorem}{Theorem}[section]
\newtheorem{lemma}[theorem]{Lemma}
\newtheorem{proposition}[theorem]{Proposition}

\theoremstyle{definition}
\newtheorem{definition}[theorem]{Definition}

\newtheorem{example}[theorem]{Example}
\newtheorem{hypothesis}[theorem]{Hypothesis}

\newtheorem{para}[theorem]{}
\newtheorem{question}[theorem]{Question}
\newtheorem{setup}[theorem]{Setup}

\theoremstyle{remark}
\newtheorem{remark}[theorem]{Remark}

\numberwithin{equation}{section}

\title{The (ir)regularity of Tor and Ext}
\author[M.~Chardin]{Marc Chardin}
\address{Institut de math\'ematiques de Jussieu, CNRS \& Sorbonne Universit\'e, 4 place Jussieu, 75005 Paris , France}
\email{marc.chardin@imj-prg.fr}

\author[D.~Ghosh]{Dipankar Ghosh}
\address{Department of Mathematics, Indian Institute of Technology Hyderabad, Kandi, Sangareddy, Telangana - 502285, India}
\email{dghosh@iith.ac.in, dipug23@gmail.com}

\author[N.~Nemati]{Navid Nemati}
\address{Institut de math\'ematiques de Jussieu, Sorbonne Universit\'e, 4 place Jussieu, 75005 Paris , France}
\email{navid.nemati@imj-prg.fr}

\date{May 6, 2019}

\begin{document}

\pagenumbering{arabic}
\thispagestyle{empty} 
\maketitle
  
 \begin{abstract}
 	We investigate the asymptotic behavior of Castelnuovo-Mumford regularity of Ext and Tor, with respect to the homological degree, over complete intersection rings. We derive from a theorem of Gulliksen a  linearity result for the regularity of Ext modules in high homological degrees. We show a  similar result for Tor, under the additional hypothesis that high enough Tor modules are supported in dimension at most one; we then provide examples showing that  the behavior could be pretty hectic when the latter condition is not satisfied.
 \end{abstract}

\section{Introduction}
	
	There has been a keen interest in understanding the behavior of $ \reg(I^n) $ as a function of $ n $, where $I$ is a homogeneous ideal in a polynomial ring $ Q = K[X_1,\ldots, X_d] $ over a field. Geramita, Gimigliano and Pitteloud \cite{GGP95} and Chandler \cite{Cha97} proved that if $ \dim(Q/I) \le 1 $, then $ \reg(I^n) \le n \cdot \reg(I) $ for all $ n \ge 1 $. This bound need not be true for higher dimension, due to an example of Sturmfels \cite{Stu00}. However, in \cite[Thm.~3.6]{Swa97}, Swanson showed that $ \reg(I^n) \le kn $ for all $ n \ge 1 $, where $ k $ is some constant. Thereafter, Cutkosky, Herzog and Trung \cite[Thm. 1.1]{CHT99} and Kodiyalam \cite{Kod00} independently showed that asymptotically $ \reg(I^n) $ is a linear function of $ n $. Later, in \cite[Thm. 3.2]{TW05}, Trung and Wang generalized this result over Noetherian standard graded ring. This behavior also has been studied for powers of more than one ideals in \cite{BCH13}, \cite{Gho16} and \cite{BC17}.
	
	One notices that $ \Tor_1^Q(Q/I^p, Q/I^q)=I^p/I^{p+q}$ if $p\geq q$, which relates this question to more general results for finitely generated graded $ Q $-modules $ M $ and $ N $. The following results are known in this case.

\begin{enumerate}[{\rm (1)}]
	\item \cite[Thm. 5.7]{Cha07} If $ \dim(\Tor_1^Q(M, N)) \le 1 $, then
	\[
		\max_{0 \le i \le d}\{ \reg\big( \Tor_i^Q(M, N) \big) - i \} = \reg(M) + \reg(N).
	\]
	This generalizes results of 
	Sidman \cite{Sid02}, Conca-Herzog \cite{CH03}, Caviglia \cite{Cav07} and Eisenbud-Huneke-Ulrich \cite[Cor.~3.1]{EHU06}. The equality in (1) extends to the case when $ Q $ is standard graded, and $ M $ or $ N $ has finite projective dimension, replacing the right hand side by $ \reg(M) + \reg(N) - \reg(Q) $.
	\item \cite[Thm. 3.2 and 4.6]{CD08} 
		\[
	\min_{0 \le i \le d}\{ \indeg (\Ext^i_Q(M,N))+i\} =\indeg (N)-\reg (M),
		\]
	 and if $ \dim (M \otimes_Q N) \le 1 $, then
	\[
	\max_{0 \le i \le d} \left\{ \reg\big( \Ext_Q^i(M,N) \big) + i \right\} = \reg(N) - \mathrm{indeg}(M),
	\]
	where $ \mathrm{indeg}(M) := \inf\{ n \in \mathbb{Z} : M_n \neq 0 \} $.
	\item \cite[ Thm.~2.4(2) and 3.5]{CHH11} An upper bound of $ \reg( \Ext_Q^i(M,N) ) + i $ is given in terms of certain invariants of $ M $ and $ N $.
\end{enumerate}

When working over standard graded algebras that are not regular ({\it i. e.} not a polynomial ring over a regular ring), one can also bound regularity of Tor modules under the same kind of hypothesis, for instance the following theorem, which follows along the same lines as in the proof of \cite[Thm.~5.17]{Cha07}.
	
\begin{theorem}\label{thm:Chardin}
	Suppose $ Q $ is a standard graded ring over a field, but $ Q $ is not a polynomial ring. Let $ M $ and $ N $ be finitely generated graded $ Q $-modules, and $ d := \min\{ \dim(M), \dim(N) \} $. If $ \dim\big( \Tor_i^Q(M, N) \big) \le 1 $ for all $ i \ge i_0 $, then
	\[
		\reg \left( \Tor_i^Q(M, N) \right) -i \le \reg(M) + \reg(N) + \left\lfloor \frac{i+d}{2} \right\rfloor (\reg(Q) -1),\ \forall i \ge i_0,
	\]
	{\rm (}and for $i=i_0-1$ if $i_0=1${\rm )}.
\end{theorem}

This implies that if $ \Proj(Q) $ has isolated singularities, then the estimate in Theorem~\ref{thm:Chardin} holds true for $ i \ge \dim(Q)-1 $.

Over complete intersection ring, the following result controls the asymptotic behavior with respect to both a power of an ideal and the homological degree.

\begin{theorem}\cite[Thm.~5.4]{GP18}\label{thm:Ghosh-Puthen}
	Set $ A := Q/({\bf f}) $, where $ Q $ is a polynomial ring over a field, and $ {\bf f} = f_1,\ldots,f_c $ is a homogeneous $ Q $-regular sequence. Let $ M $ and $ N $ be finitely generated graded $ A $-modules, and $ I $ be a homogeneous ideal of $ A $. Then, 
	\begin{enumerate}[{\rm (i)}]
		\item 
		$ 
		\reg\left( \Ext_A^{i}(M, I^nN) \right) \le \rho_N(I) \cdot n - w \cdot \left\lfloor \frac{i}{2} \right\rfloor + e,\quad \forall i, n \ge 0 ,
		$
		\item 
		$ 
		\reg\left( \Ext_A^{i}(M,N/I^nN) \right) \le \rho_N(I) \cdot n - w \cdot \left\lfloor \frac{i}{2} \right\rfloor + e',\quad \forall i, n \ge 0,
		$
	\end{enumerate}
	where $ e,e'\in \mathbb{Z}$, $ w := \min\{ \deg(f_j) : 1 \le j \le c \} $, and $ \rho_N(I) $ is an invariant defined in terms of reduction ideals of $ I $ with respect to $ N $.
\end{theorem}

Moreover, in \cite[6.6]{GP18}, Ghosh and Puthenpurakal raised the following question. 

\begin{question}\label{question2}
	For $ \ell \in \{0,1\} $, do there exist $ a_\ell , a'_\ell \in \mathbb{Z}_{>0} $ and $ e_\ell , e'_\ell \in \mathbb{Z} \cup \{ -\infty \} $ such that
	\begin{enumerate}[{\rm (i)}]
		\item $ \reg\big(\Ext_A^{2i+\ell}(M,N)\big) = - a_\ell \cdot i + e_\ell $ for all $ i \gg 0 $ ?
		\item $ \reg\left(\Tor_{2i+\ell}^A(M,N)\right) = a'_\ell \cdot i + e'_\ell $ for all $ i \gg 0 $ ?
	\end{enumerate}
\end{question}

	In this text, we are addressing these questions. We prove that the answer to (i) is positive, even in a more general situation, while the answer to (ii) is negative. However, if $ \dim\left( \Tor_i^A(M, N) \right) \le 1 $ for all $ i \gg 0 $, the second question does have a positive answer.
	
	Our main positive result on these questions is the following:
	
	\begin{headtheorem}[Theorems~\ref{thm:main-Ext} and \ref{cor:lin-reg-Tor-dim-1}] \label{thm:main-Intro}
	Let $ Q $ be a standard graded Noetherian algebra, $ A := Q/({\bf f}) $, where $ {\bf f}  := f_1,\ldots,f_c $ is a homogeneous $ Q $-regular sequence. Let $ M $ and $ N $ be finitely generated graded $ A $-modules such that $ \Ext_Q^i(M,N) = 0 $ for all $ i \gg 0 $.
	
	Then, 
	\begin{enumerate}[{\rm (i)}]
	\item for every $ \ell \in \{0,1\} $, there exist $ a_{\ell} \in \{ \deg(f_j) : 1 \le j \le c \} $ and $ e_{\ell} \in \mathbb{Z} \cup \{ - \infty \} $ such that
	\[
		\reg \big( \Ext_A^{2i+\ell}(M,N) \big) = - a_{\ell} \cdot i + e_{\ell} \quad \mbox{for all } i \gg 0.
	\]
	\item if further $Q$ is *local or the epimorphic image of a Gorenstein ring, $M$ has finite projective dimension over $Q$ and 
	$$
	\dim\left( \Tor_i^A(M,N)\right)\le 1,\ \forall i \gg 0,
	$$
	then, for every $ \ell \in \{0,1\} $, there exist $ a'_\ell \in \{ \deg(f_j) : 1 \le j \le c \} $ and $ e'_\ell \in \mathbb{Z} \cup \{ - \infty \} $ such that	
	$$
	\reg\left( \Tor_{2i+\ell}^A(M,N) \right) = a'_\ell \cdot i + e'_\ell ,\ \forall i \gg 0.
	$$
	\end{enumerate}
\end{headtheorem}

On the negative side, we provide examples showing that the behavior of the regularity of Tor modules could be very different without the assumptions as in the result above.

\begin{headexample}[Example \ref{exam:reg-Tor-Ext-setup1}]\label{exam:reg-Tor-Ext-Ex1}
 	Let $ Q := K[Y,Z,V,W] $ be a polynomial ring with usual grading over a field $ K $, and $ A := Q/(Y^2,Z^2) $. Write $ A = K[y,z,v,w] $, where $ y, z, v $ and $ w $ are the residue classes of $ Y,Z,V $ and $ W $ respectively. Fix an integer $ m \ge 1 $. Set
 	\[
 	M := \Coker \left( \begin{bmatrix}
 	y & z & 0 & 0 \\
 	-v^m & -w^m & y & z
 	\end{bmatrix} : {\begin{array}{c}
 		A(-m)^2\\
 		\bigoplus \\
 		A(-1)^2
 		\end{array}} \longrightarrow {\begin{array}{c}
 		A(-m+1)\\
 		\bigoplus \\
 		A
 		\end{array}} \right)
 	\]
 	and $ N := A/(y,z) $. Then, for every $ i \ge 1 $, we have
 	\begin{enumerate}[{\rm (i)}]
 		\item $ \indeg\left( \Ext_A^i(M,N)  \right) = -i - m+1  $ and $ \reg\left( \Ext_A^i(M,N)  \right) = - i $.
 		\item $ \indeg\left( \Tor_i^A(M,N)  \right) = i $ and $ \reg\left( \Tor_i^A(M,N)  \right) = (m+1)i + (2m-2) $.		
 	\end{enumerate}
 \end{headexample}

	In this example, $\Tor_i^A(M,N)$  is supported in dimension 2 for $i\gg0$, its regularity is eventually linear, but the leading term depends on the module $M$ and could be arbitrarily large, opposite to the case where $\Tor_i^A(M,N)$  is supported in dimension 1 for $i\gg0$ -- in that case we showed the leading term would then be ${{2}\over{2} }=1$, as compared to $m+1$ here.
	
	This shows that the finiteness result for the $\Tor$-algebra that we prove under the condition that  $\Tor_i^A(M,N)$  is supported in dimension 1 for $i\gg0$ can fail if this hypothesis is removed. Additional results around the hypothesis on the asymptotic dimension of $\Tor$ are given in Remark \ref{rmk:thm:main-gen-version} and in Proposition \ref{prop:zero-two-loc-coh}.
	
	The following example that we develop in the last section shows that the eventual regularity of $\Tor$ could be very far from being linear,
	
	\begin{headexample}[Example~\ref{exam:reg-Tor-Ext-setup2}]
	Let $ Q := K[X,Y,Z,U,V,W] $ be a standard graded polynomial ring over a field $ K $ of characteristic $ 2 $, and $ A := Q/(X^2,Y^2,Z^2) $. We write $ A = K[x,y,z,u,v,w] $, where $ x,y,z,u,v $ and $ w $ are the residue classes of $ X,Y,Z,U,V $ and $ W $ respectively. Set
	\[
	M := \Coker \left( \begin{bmatrix}
	x & y & z & 0 & 0 & 0 \\
	u & v & w & x & y & z 
	\end{bmatrix} : A(-1)^6 \longrightarrow A^2 \right)
	\quad \mbox{and} \quad N:= A/(x,y,z).
	\]
	Then, for every $ n \ge 1 $, we have
	\begin{enumerate}[{\rm (i)}]
		\item $ \indeg\left( \Ext_A^n(M,N)  \right) =  \reg\left( \Ext_A^n(M,N)  \right) = - n $.
		\item $ \indeg\left( \Tor_n^A(M,N)  \right) = n $ and $ \reg\left( \Tor_n^A(M,N)  \right) = n+f(n) $, where
		\[
		f(n) := \left\{ \begin{array}{ll}
		2^{l+1} - 2 & \mbox{if } n = 2^l - 1 \\
		2^{l+1} - 1 & \mbox{if } 2^l \le n \le 2^{l+1} - 2
		\end{array}\right. \quad \mbox{for all integers $ l \ge 1 $}.
		\]
	\end{enumerate}
\end{headexample}

As a consequence, in this example,
\[ \{\reg(\Tor_{2n}^A(M,N))/(2n) : n \ge 1\} \quad \mbox{and} \quad \{\reg(\Tor_{2n+1}^A(M,N))/(2n+1) : n \ge 1\}
	\]
	are dense sets in $ [2,3] $ and 
	\[
		\liminf_{n \to \infty } \dfrac{\reg(\Tor_n^A(M,N))}{n} = 2 \quad \mbox{and} \quad \limsup_{n \to \infty } \dfrac{\reg(\Tor_n^A(M,N))}{n} = 3.
	\]
	
\section{Module structures on Ext  and Tor}\label{sec:module-structures-Ext-Tor}

Most of our results are proved under the following hypothesis.

\begin{hypothesis}\label{hyp:1}
	The ring $ Q $ is a standard graded Noetherian algebra, $ A = Q/({\bf f}) $, where $ {\bf f}  := f_1,\ldots,f_c $ is a homogeneous $ Q $-regular sequence with $ w_j = \deg(f_j) $, and $ M , N $ are finitely generated graded $ A $-modules such that $ \Ext_Q^i(M,N) = 0 $ for all $ i \gg 0 $.
\end{hypothesis}

\begin{para}\label{para:Matlis-duality}
	Write $ A = A_0[x_1,\ldots,x_d] $, where $ \deg(x_i) = 1 $ for $ 1 \le i \le d $. When $ (A_0,\mathfrak{m}_0) $ is local, then following the terminologies in \cite[pp.~141]{BH98}, $ A $ is *local, i.e., it has a unique maximal homogeneous ideal $\mathfrak{m}=\mathfrak{m}_0+A_+$. Setting $ E_0 := E_{A_0}(A_0/\mathfrak{m}_0) $, the Matlis dual of $ M $ is defined to be $ M^{\vee} := {\text *}\Hom_{A_0}(M,E_0) $, where $ \left( M^{\vee} \right)_n = \Hom_{A_0}(M_{-n},E_0) $ for every $ n \in \mathbb{Z} $. In view of \cite[Prop. 3.6.16 and Thm. 3.6.17]{BH98}, the contravariant functor $ (-)^{\vee} $ from the category of finitely generated graded $ A $-modules to itself is exact, and $ M^{\vee \vee} \cong M $.
\end{para}

\begin{para}[\bf Eisenbud operators]\label{Eisenbud-operators}
	We need to remind facts about Eisenbud operators \cite[Section~1]{Eis80} in the graded setup. By a {\it homogeneous homomorphism}, we mean a graded homomorphism of degree zero. Let $ \mathbb{F} : ~ \cdots \rightarrow F_n \rightarrow \cdots \rightarrow F_1\rightarrow F_0 \rightarrow 0 $ be a graded free resolution of $ M $ over $ A $. In view of the construction of Eisenbud operators \cite[pp.~39, (b)]{Eis80}, one may choose homogeneous $ A $-module homomorphisms $ t'_j : F_{i+2} \to F_i(-w_j) $ (for every $ i $) corresponding to $ f_j $.

	Thus the {\it Eisenbud operators} corresponding to $ {\bf f}  = f_1,\ldots,f_c $ are given by $ t'_j : \mathbb{F}[2] \rightarrow \mathbb{F}(-w_j) $, $ 1 \le j \le c $, where $ [-] $ and $ (-) $ denote respectively shift in homological degree and internal degree.
\end{para}

\begin{para}[\bf Graded module structures on Ext and Tor]\label{grad-mod-struc-Ext-Tor}
	The homogeneous chain maps $t'_j$ are determined uniquely up to homotopy; see \cite[Cor. 1.4]{Eis80}. Therefore the maps
	\begin{align*}
	\Hom_A(t'_j,N) & : \Hom_A(\mathbb{F}(-w_j),N) \longrightarrow \Hom_A(\mathbb{F}[2],N) \\
	t'_j \otimes_A 1_N & : \mathbb{F}[2] \otimes_A N \longrightarrow \mathbb{F}(-w_j) \otimes_A N
	\end{align*}
	induce well-defined homogeneous $ A $-module homomorphisms
	\begin{align}
		s_j & : \Ext_A^i(M,N) \longrightarrow \Ext_A^{i+2}(M,N)(-w_j) \quad \mbox{for all $ i $ and $ 1 \le j \le c $}, \label{Eis-operators-on-ext}\\
		t_j & : \Tor_{i+2}^A(M,N) \longrightarrow \Tor_i^A(M,N)(-w_j) \quad \mbox{for all $ i $ and $ 1 \le j \le c $}. \label{Eis-operators-on-tor}
	\end{align}
	Hence, for every $ l \ge 0 $, applying the functors $ H_{A_+}^l(-) $ and $ (-)^{\vee} $ successively on \eqref{Eis-operators-on-tor}, one obtains the homogeneous $ A $-module homomorphisms
	\begin{equation}\label{Eis-operators-on-tor-plus-coh}
		^+{t_j^l} := H_{A_+}^l(t_j)^{\vee} : H_{A_+}^l \hspace{-0.1cm} \left( \Tor_i^A(M,N) \right)\hspace{-0.1cm}^{\vee} \longrightarrow H_{A_+}^l \hspace{-0.1cm} \left( \Tor_{i+2}^A(M,N) \right)\hspace{-0.1cm}^{\vee}(-w_j)
	\end{equation}
	\begin{equation}\label{Eis-operators-on-tor-loc-coh}
		^\mathfrak{m}{t_j^l} := H_{\mathfrak{m}}^l(t_j)^{\vee} : H_{\mathfrak{m}}^l \hspace{-0.1cm} \left( \Tor_i^A(M,N) \right)\hspace{-0.1cm}^{\vee} \longrightarrow H_{\mathfrak{m}}^l \hspace{-0.1cm} \left( \Tor_{i+2}^A(M,N) \right)\hspace{-0.1cm}^{\vee}(-w_j)
	\end{equation}
	for all $ i $ and $ 1 \le j \le c $. These coincide whenever $A_0$ is artinian.
	By \cite[Cor. 1.5]{Eis80}, since the chain maps $ t'_j $ $ ( 1 \le j \le c ) $ commute up to homotopy,
	$$
	 \Ext_A^{\star}(M,N) ,
	\ H_{A_+}^l \hspace{-0.1cm} \left( \Tor_{\star}^A(M,N)\right) \hspace{-0.1cm}^{\vee} 
	\ {\rm and}\ H_{\mathfrak{m}}^l \hspace{-0.1cm} \left( \Tor_{\star}^A(M,N)\right) \hspace{-0.1cm}^{\vee} 
	$$
	turn into graded $ T := A[y_1,\ldots,y_c] $-modules, where $ T $ is a graded polynomial ring  over $ A $ with $ \deg(y_j) = 2 $ for $ 1 \le j \le c $. The actions of $ y_j $ on these three graded $T$-modules are defined by the maps $ s_j $,  $ ^+{t_j^l} $ and $ ^\mathfrak{m}{t_j^l} $, respectively.

	These structures depend only on $ {\bf f} $, are natural in both module arguments and commute with the connecting maps induced by short exact sequences.
	
	Choosing a graded epimorphism $B\rightarrow Q$, such that $B$ is *local and Cohen-Macaulay of dimension $b$, with canonical module $\omega_B$, local duality provides a commutative diagram,
	$$
	\xymatrix{
	H_{\mathfrak{m}}^l \hspace{-0.1cm} \left( \Tor_i^A(M,N) \right)\hspace{-0.1cm}^{\vee} \ar^(.45){^\mathfrak{m}{t_j^l}}[rr] \ar^{\simeq}[d]&&
	H_{\mathfrak{m}}^l \hspace{-0.1cm} \left( \Tor_{i+2}^A(M,N) \right)\hspace{-0.1cm}^{\vee}(-w_j)\ar^{\simeq}[d]\\
	\Ext^{b-l}_B\left(\Tor_i^A(M,N),\omega_B\right) \ar^(.44){\Ext_B^{b-l}(t_j,\omega_B)}[rr] &&\Ext^{b-l}_B\left(\Tor_{i+2}^A(M,N),\omega_B\right)(-w_j) \\}
	$$
	where the map on the top row identifies to the one in \ref{Eis-operators-on-tor-plus-coh}, whenever $A_0$ is artinian.
\end{para}

\begin{theorem}\cite[Thm. 3.1]{Gul74}\label{thm:Gulliksen}
	The graded module $ \Ext_A^{\star}(M,N) $ is finitely generated over $ A[y_1,\ldots,y_c] $ provided $ \Ext_Q^i(M,N) = 0 $ for all $ i \gg 0 $.
\end{theorem}

For instance, when $ Q $ is a polynomial ring over a field, $ \Ext_A^{\star}(M,N) $ is finitely generated over $ A[y_1,\ldots,y_c] $, but $ H_{A_+}^l \hspace{-0.1cm} \left( \Tor_{\star}^A(M,N)\right) \hspace{-0.1cm}^{\vee} $ is not necessarily finitely generated by Remark~\ref{rmk:fg-tor}. Nevertheless, we prove that if $ \dim(\Tor_i^A(M,N)) \le 1 $ for all $ i \gg 0 $, then the modules $ H_{\mathfrak{m}}^l \hspace{-0.1cm} \left( \Tor_{\star}^A(M,N)\right) \hspace{-0.1cm}^{\vee} $ are finitely generated over $ A[y_1,\ldots,y_c] $; see Theorem~\ref{thm:main}. In order to prove our results, we use the canonical bigraded structures on these graded modules.

\begin{para}[\bf Bigraded structures]\label{bigrad-mod-struc-Ext-Tor}
	We make $ T = A[y_1,\ldots,y_c] $ a $ \mathbb{Z}^2 $-graded ring as follows. Write
	\begin{equation}\label{Z2-graded-ring}
		T = A[y_1,\ldots,y_c] = A_0[x_1,\ldots,x_d,y_1,\ldots,y_c],
	\end{equation}
	and set $ \deg(x_i) = (0,1) $ for $ 1 \le i \le d $ and $ \deg(y_j) = (2,-w_j) $ for $ 1 \le j \le c $. We give $ \mathbb{Z}^2 $-grading structures on $ E^\star :=\Ext_A^{\star}(M,N) $, $ ^+{D^l_\star}:=H_{A_+}^l \hspace{-0.1cm} \left( \Tor_{\star}^A(M,N)\right) \hspace{-0.1cm}^{\vee} $  
	and $ ^\mathfrak{m}{D^l_\star}:=H_{\mathfrak{m}}^l \hspace{-0.1cm} \left( \Tor_{\star}^A(M,N)\right) \hspace{-0.1cm}^{\vee} $ by setting their $ (i,a) $th graded components as the $ a $th graded components of $ \mathbb{Z} $-graded modules $ \Ext_A^i(M,N) $, $H_{A_+}^l \hspace{-0.1cm} \left( \Tor_i^A(M,N)\right) \hspace{-0.1cm}^{\vee} $ and $ H_{\mathfrak{m}}^l \hspace{-0.1cm} \left( \Tor_i^A(M,N)\right) \hspace{-0.1cm}^{\vee} $ respectively, for $ (i,a) \in \mathbb{Z}^2 $. Hence, in view of Section~\ref{grad-mod-struc-Ext-Tor},  $ E^\star$ $ ^+{D^l_\star}$ and  $^\mathfrak{m}{D^l_\star}$ are $ \mathbb{Z}^2 $-graded $ T $-modules. We consider the graded submodules corresponding to direct sums of even and odd components :
	\begin{align}
		E^{2\star} := \bigoplus_{i \in \mathbb{Z}} \Ext_A^{2i}(M,N),\quad E^{2\star +1} := \bigoplus_{i \in \mathbb{Z}} \Ext_A^{2i+1}(M,N),
		\label{even-odd-ext-mods}
	\end{align}
	that we will also refer to as  $ \Ext_A^{\rm even}(M,N) $ and $ \Ext_A^{\rm odd}(M,N) $, respectively,  depending on the context.
Similarly, one defines
	\begin{align}
^+{D^l_{2\star}}, \; ^+{D^l_{2\star +1}},  \; ^\mathfrak{m}{D^l_{2\star}},  \; ^\mathfrak{m}{D^l_{2\star +1}}
 \label{odd-loc-coh-tor-mods}
	\end{align}
by taking direct sums over even or odd homological degree components.

	In view of \eqref{Z2-graded-ring}, set a polynomial ring $ S := Q_0[X_1,\ldots,X_d,Y_1,\ldots,Y_c] $, where $ \deg(X_i) = (0,1) $ for $ 1 \le i \le d $ and $ \deg(Y_j) = (1,-w_j) $ for $ 1 \le j \le c $. The modules stated in \eqref{even-odd-ext-mods} and \eqref{odd-loc-coh-tor-mods} are canonically $ \mathbb{Z}^2 $-graded $ S $-modules. For instance, the $ (i,a) $th graded component of $E^{2\star}$ is defined to be $ \Ext_A^{2i}(M,N)_a $ for $ (i,a) \in \mathbb{Z}^2 $, while the actions of $ X_1,\ldots,X_d,Y_1,\ldots,Y_c $ on $ E^{2\star}$ are defined by $ x_1,\ldots,x_d,y_1,\ldots,y_c $ respectively. Note that $ Y_j \cdot \Ext_A^{2i}(M,N) \subseteq \Ext_A^{2(i+1)}(M,N) $ for $ i \in \mathbb{Z} $ and $ 1 \le j \le c $.
\end{para}

Thus, in bigraded setup, we have the following result on Ext modules.

\begin{proposition}\label{prop:Gulliksen-cor}
	If $ \Ext_Q^i(M,N) = 0 $ for all $ i \gg 0 $, then $ \Ext_A^{\rm even}(M,N) $ and $ \Ext_A^{\rm odd}(M,N) $ are finitely generated $ \mathbb{Z}^2 $-graded over $ S = Q_0[X_1,\ldots,X_d,Y_1,\ldots,Y_c] $, where $ \deg(X_l) = (0,1) $ for $ 1 \le l \le d $ and $ \deg(Y_j) = (1,-w_j) $ for $ 1 \le j \le c $. 
	\end{proposition}

Recall that for every $ i,a \in \mathbb{Z} $,
	\[
		\Ext_A^{\rm even}(M,N)_{(i,a)} = \Ext_A^{2i}(M,N)_a \quad \mbox{and} \quad \Ext_A^{\rm odd}(M,N)_{(i,a)} = \Ext_A^{2i+1}(M,N)_a,
	\]
	where $ L_{(i,*)} := \bigoplus_{a \in \mathbb{Z}} L_{(i,a)} $ for a $ \mathbb{Z}^2 $-graded $ S $-module $ L $.

\begin{proof}
	By virtue of Theorem~\ref{thm:Gulliksen}, $ \Ext_A^{\star}(M,N) $ is a finitely generated graded module over $ T = A[y_1,\ldots,y_c] $. Therefore the graded submodules $ \Ext_A^{\rm even}(M,N) $ and $ \Ext_A^{\rm odd}(M,N) $ are also finitely generated. Since we are only extending the grading, the proposition now follows from \ref{bigrad-mod-struc-Ext-Tor}.
\end{proof}

\section{Linearity of regularity of Ext and Tor}\label{sec:lin-reg-ext-tor}
	
	In this section, we show that $ \reg( \Ext_A^{2i}(M,N) ) $ and $ \reg( \Ext_A^{2i+1}(M,N) ) $ are asymptotically linear in $ i $, where $ M $ and $ N $ are finitely generated graded modules over a graded complete intersection ring $ A $. Moreover, a similar result for Tor modules is proved when $ \dim(\Tor_i^A(M,N)) \le 1 $ for all $ i \gg 0 $. We use the following result, which is a consequence of a theorem due to Bagheri, Chardin and H\`a.
	
\begin{proposition}\cite[Thm.~4.6]{BCH13}\label{prop:BCH}
	Let $ Q_0 $ be a commutative Noetherian ring. Set $ R := Q_0 [X_1,\ldots,X_d,Z_1,\ldots,Z_c] $, where $ \deg(X_i) = (0,1) $ for $ 1 \le i \le d $ and $ \deg(Z_j) = (1,g_j) $ for some $ g_j \in \mathbb{Z} $, $ 1 \le j \le c $. Let $ L  $ be a finitely generated $ \mathbb{Z}^2 $-graded $ R $-module. Set $ Q := Q_0[X_1,\ldots,X_d] $, where $ \deg(X_i) = 1 $ for $ 1 \le i \le d $.
	
	Then, for every $ l \ge 0 $, there exist $ a_l, a'_l \in \{ g_j : 1 \le j \le c \} $, $ e_l \in \mathbb{Z} \cup \{ - \infty \} $ and $ e'_l \in \mathbb{Z} \cup \{ + \infty \} $ such that
	\begin{align}
		\eend\big( \Tor_l^Q(L_{(t,*)},Q_0) \big) & = t \cdot a_ l+ e_l \quad \mbox{for all } t \gg 0, \label{end-tor-lin}\\
		\indeg\big( \Tor_l^Q(L_{(t,*)},Q_0) \big) & = t \cdot a'_l + e'_l \quad \mbox{for all } t \gg 0,\label{indeg-tor-lin}
	\end{align}
	where $ \mathrm{end}(M) := \sup\{ n \in \mathbb{Z} : M_n \neq 0 \} $ and $ \mathrm{indeg}(M) := \inf\{ n \in \mathbb{Z} : M_n \neq 0 \} $ for a graded $ Q $-module $ M $.
	Hence, there exist $ a, a' \in \{ g_j : 1 \le j \le c \} $, $ e \in \mathbb{Z} \cup \{ - \infty \} $ and $ e' \in \mathbb{Z} \cup \{ + \infty \} $ such that
	\begin{align*}
		& \reg\left( L_{(t,*)} \right) = \sup\left\{ \eend\left( \Tor_l^Q(L_{(t,*)},Q_0) \right) - l : 0 \le l \le d \right\} = t \cdot a + e \mbox{ for all } t \gg 0,\\
		& \indeg\left( L_{(t,*)} \right) = \indeg\big( \Tor_0^Q(L_{(t,*)},Q_0) \big) = t \cdot a' + e' \mbox{ for all } t \gg 0.
	\end{align*}
\end{proposition}

\begin{proof}
	The same proof as of \cite[Thm.~4.6]{BCH13} works if one considers $ L $ in place of $ M\mathcal{R} $. We use notations as in this reference, in particular $\Supp_{\mathbb{Z}}(M)=\{ \mu\in \mathbb{Z}\ \vert\ M_\mu\not= 0\}$ and for two tuples $c=(c_1,\ldots ,c_s)\in \mathbb{Z}_{\ge 0}^s$ and $E=(e_1,\ldots ,e_s)$ with $e_i$ in an abelian group ($\mathbb{Z}$-module), $|c|:=c_1+\cdots +c_s$,  $|E|=s$ and $c.E:=c_1e_1+\cdots +c_s e_s$.

	As in \cite[Thm.~4.6]{BCH13}, there exist a finite collection of integers $ \{ \delta_p^l, t_{p,1}^l : 1 \le p \le m \} $, and tuples $ E_{p,1}^l $ of elements in $ \Gamma_1 = \{ g_j : 1 \le j \le c \} $ such that $ \Delta E_{p,1}^l $ is linearly independent for every $ 1 \le p \le m $, satisfying:
	\begin{equation}\label{betti-shape}
		\Supp_{\mathbb{Z}} \left( \Tor_l^Q(L_{(t,*)},Q_0) \right) = \bigcup_{p = 1}^m \Big( \delta_p^l + \bigcup_{c_1 \in \mathbb{Z}_{\ge 0}^{|E_{p,1}^l|}, |c_1| = t - t_{p,1}^l} c_1 . E_{p,1}^l \Big)
	\end{equation}
	for all $ t \ge \max_p\{ t_{p,1}^l \} $, where $ \Delta E_{p,1}^l = ( h_2 - h_1, \ldots, h_r - h_{r-1} ) $ if $ E_{p,1}^l = (h_1,\ldots,h_r)$ and $r\geq 2$, and else empty. So the cardinality of each $ E_{p,1}^l $ must be at most $ 2 $.
	It can be observed that the equalities \eqref{end-tor-lin} and \eqref{indeg-tor-lin} follow from \eqref{betti-shape} once we set
	\begin{align*}
		a_l & := \max\{ h : h \in E_{p,1}^l, 1 \le p \le m \},\quad a'_l := \min\{ h : h \in E_{p,1}^l, 1 \le p \le m \},\\
		e_l & := \max\{ \delta_p^l - a_l \cdot t_{p,1}^l : 1 \le p \le m \mbox{ for which } a_l \in E_{p,1}^l \} \quad \mbox{and} \\
		e'_l & := \min\{ \delta_p^l - a_l \cdot t_{p,1}^l : 1 \le p \le m \mbox{ for which } a_l \in E_{p,1}^l \}.
	\end{align*}
	Finally, one obtains the last part from \eqref{end-tor-lin} and \eqref{indeg-tor-lin} by choosing suitable $ a, a', e $ and $ e' $.
\end{proof}

Here are our results on the linearity of regularity for Ext and Tor modules.

\begin{theorem}\label{thm:main-Ext}
	Let $ Q $ be a standard graded Noetherian algebra, $ A := Q/({\bf f}) $, where $ {\bf f}  := f_1,\ldots,f_c $ is a homogeneous $ Q $-regular sequence. Let $ M $ and $ N $ be finitely generated graded $ A $-modules such that $ \Ext_Q^i(M,N) = 0 $ for all $ i \gg 0 $.
	
	Then, for every $ \ell \in \{0,1\} $, there exist $ a_{\ell} \in \{ \deg(f_j) : 1 \le j \le c \} $ and $ e_{\ell} \in \mathbb{Z} \cup \{ - \infty \} $ such that
	\[
		\reg \big( \Ext_A^{2i+\ell}(M,N) \big) = - a_{\ell} \cdot i + e_{\ell} \quad \mbox{for all } i \gg 0.
	\]
\end{theorem}

\begin{proof}
	The theorem follows from Propositions~\ref{prop:Gulliksen-cor} and \ref{prop:BCH}.
\end{proof}

\begin{remark}
More precisely, for every $ \ell \in \{0,1\} $, Proposition \ref{prop:BCH} shows that, for any $j$, the initial and ending degrees of $ \Tor^{Q_0[X_1,\dots,X_d]}_j \big( \Ext_A^{2i+\ell}(M,N),Q_0 \big) $ are eventually  linear functions in $ i $.
\end{remark}

\begin{remark}
	In Theorem~\ref{thm:main-Ext}, if $ Q $ is regular, then the assumption on vanishing of Ext modules over $ Q $ is superfluous.	
\end{remark}

	The asymptotic linearity of regularity for Tor modules holds in certain cases.
	
	\begin{theorem}\label{cor:lin-reg-Tor-dim-1}
		Let $ Q $ be a standard graded Noetherian algebra, $ A := Q/({\bf f}) $, where $ {\bf f}  := f_1,\ldots,f_c $ is a homogeneous $ Q $-regular sequence. Assume $Q$ is *local or the epimorphic image of a Gorenstein ring. Let $ M $ and $ N $ be finitely generated graded $ A $-modules such that,\par
		
		{\rm (i)}  $M$ has finite projective dimension over $Q$,\par
		{\rm (ii)} $\dim \big(\Tor_i^A(M,N) \big) \le 1 $ for any $ i \gg 0 $.\par
		
		Then, for every $ \ell \in \{0,1\} $, there exist $ a_\ell \in \{ \deg (f_j) : 1 \le j \le c \} $ and $ e_\ell \in \mathbb{Z} \cup \{ - \infty \} $ such that
		$$ \reg\left( \Tor_{2i+\ell}^A(M,N) \right) = a_\ell \cdot i + e_\ell ,\ \forall i \gg 0.$$
\end{theorem}

We postpone the proof of Theorem~\ref{cor:lin-reg-Tor-dim-1} until presenting ingredients of the proof. 

\begin{remark}
In Sections~\ref{sec:examples:linearity} and \ref{sec:examples:nonlinearity}, we show that the condition (ii) in the Theorem~\ref{cor:lin-reg-Tor-dim-1} cannot be omitted.
\end{remark}

\begin{lemma}\label{lem:reg_and_regm}
Let $B\rightarrow A$ be a graded epimorphism of *local rings. Assume that $ B $ is Cohen-Macaulay.  Let $W$ be a finitely generated  graded $A$-module. Set $\reg_{\mathfrak{m}}(W):= \max_{j} \lbrace \eend(H^j_{\mathfrak{m}}(W))+j\rbrace$. Then one has
\begin{itemize}
\item[$(1)$] $\eend\big(H^i_{\mathfrak{m}}(W)\big)= -\indeg\big(\Ext_B^{\dim B-i}(W, \omega_B)\big)$.
\item[$(2)$] $\reg(W)\le \reg_{\mathfrak{m}}(W)\le \reg(W)+\dim A_0$.
\item[$(3)$] If $\dim(W) \le 1$,  then $H^p_{\mathfrak{m}_0}(H^q_{A_+}(W))=0$ for $p+q>1$, and 
\begin{align*}
\reg_{\mathfrak{m}}(W)\hspace{-0.05cm}&= \hspace{-0.05cm}\max \hspace{-0.05cm} \left\{ \hspace{-0.05cm}\eend\hspace{-0.05cm} \left( H^0_{\mathfrak{m}}(W)\right)\hspace{-0.05cm}, \eend\hspace{-0.05cm}\left(\hspace{-0.05cm}H^1_{\mathfrak{m}_0}\hspace{-0.05cm}\big(H^0_{A_+}(W)\big)\hspace{-0.05cm}\right)\hspace{-0.07cm}+\hspace{-0.05cm}1, \eend\hspace{-0.05cm}\left(\hspace{-0.05cm}H^0_{\mathfrak{m}_0}\hspace{-0.05cm}\big(H^1_{A_+}(W)\big)\hspace{-0.05cm}\right)\hspace{-0.07cm}+\hspace{-0.05cm}1\right\}\\
\reg(W)&= \max  \hspace{-0.05cm} \left\{ \hspace{-0.05cm}\eend\hspace{-0.05cm}\big( H^0_{\mathfrak{m}}(W)\big),  \eend\hspace{-0.05cm}\left(\hspace{-0.05cm}H^1_{\mathfrak{m}_0}\hspace{-0.05cm}\big(H^0_{A_+}(W)\big)\hspace{-0.05cm}\right)\hspace{-0.05cm}, \eend\hspace{-0.05cm}\left(\hspace{-0.05cm}H^0_{\mathfrak{m}_0}\hspace{-0.05cm}\big(H^1_{A_+}(W)\big)\hspace{-0.05cm}\right)\hspace{-0.05cm}+1\right\}
\end{align*}
\end{itemize}
\end{lemma}
\begin{proof}
Part $(1)$ follows from \cite[Thm. 3.6.19]{BH98}.  Part $(2)$ is \cite[Prop. 3.4]{Has12}; it follows from the proof of this result that $H^p_{\mathfrak{m}_0}(H^q_{A_+}(W))=0$ for $p+q>\dim W$, which proves $(3)$. More directly, $(3)$ follows from the fact that $ H^p_{\mathfrak{m}_0}\hspace{-0.05cm}\big(H^q_{A_+}(W)\big) = 0 $ for $q\ge 2$  and for  $p\ge 2$ as $\dim (A/\ann_A(W))\le 1$, which implies that the composed functor spectral sequence $H^p_{\mathfrak{m}_0}\hspace{-0.05cm}\big(H^q_{A_+}(W)\big) \Longrightarrow H^{p+q}_{\mathfrak{m}}(W)$ abuts at step $2$.
\end{proof}

\begin{theorem}\label{thm:main}
	Let  $B\rightarrow Q$ be a graded epimorphism, $A := Q/({\bf f})$, where $ {\bf f}  := f_1,\ldots,f_c $ is a homogeneous $ Q $-regular sequence. Let $P$ be a finitely generated graded $B$-module, and $ M, N $ be finitely generated graded $A$-modules such that

{\rm (i)} $\Ext_B^q (N, P)=0$ for $q\gg0$,

{\rm (ii)} $M$ has finite projective dimension over $Q$, 

{\rm (iii)} $\exists \,r$,  $\Ext_B^q  \hspace{-0.07cm} \left(\Tor^A_i (M,N), P\right)=0$  for every $q\not\in \{ r-1,r\}$ and $i\gg 0$.

Then, for any $q$, 
	\[
		\Ext_B^q \hspace{-0.07cm} \left( \Tor^A_\star (M,N), P\right)
	\]
is a finitely generated graded $ A[y_1,\ldots , y_c] $-module.
\end{theorem}

Recall that whenever $B$ is equidimensional Cohen-Macaulay and $P=\omega_B$, then the  modules  $\Ext_B^{\dim B-i}\hspace{-0.05cm} \left( \Tor^A_\star (M,N), \omega_B\right) $ only depend upon $i$ and $ \Tor^A_\star (M,N)$ as, in the local case, these are Matlis dual to the $i$-th local cohomologies of $\Tor^A_\star (M,N)$.

With such a choice for $B$ and $P$, condition (i) is satisfied, and condition (iii) with $r=\dim B$ is equivalent to $ \dim \left( \Tor_i^A(M,N) \right) \le 1$. This will be the main case of application of this result.

Also condition (i) is always satisfied if $B$ is regular.

\begin{proof}[Proof of Theorem~\ref{thm:main}]
	Let $\mathbb{F}_{\bullet}^{M}$ be a graded minimal free resolution of $M$ over $A$, and $ \mathbb{I}^{\bullet}_{P}$ be a  graded minimal injective resolution of $P$ over $B$. Consider the double complex $ \mathbb{K}^{\bullet,\bullet} $ defined by
	\begin{equation}\label{double-complexes}
		\mathbb{K}^{p,q}:= \Hom_B \hspace{-0.07cm} \left( \mathbb{F}^{M}_{p} \otimes_A N, \mathbb{I}^{q}_{P} \right) \cong \Hom_A \hspace{-0.07cm} \left( \mathbb{F}^M_{p}, \Hom_B \hspace{-0.07cm} \left( N,\mathbb{I}^q_{P} \right) \right)
	\end{equation}
	and its associated spectral sequences. The double complexes in \eqref{double-complexes} are equalized by the natural isomorphism. Since $\Hom_A \hspace{-0.07cm} \left( \mathbb{F}_p^M,- \right) $ is an exact functor, by computing cohomology vertically,
	\[
		^{v}E^{p,q}_1 = \Hom_A  \hspace{-0.07cm} \left( \mathbb{F}^M_p, \Ext_B^q(N,P) \right) \quad \mbox{and} \quad ^{v}E_{2}^{p,q} =
		\Ext^p_A\hspace{-0.05cm}\left(M,\Ext_B^q(N,P)\right).
	\]
According to Theorem \ref{thm:Gulliksen}, condition (ii) implies that the graded $ A[y_1,\ldots , y_c] $-modules $\Ext^\star_A(M,\Ext_B^q(N,P))$ are finitely generated for every $q$. As these are zero for all but finitely many $q$ by (i), $^{v}E^{\star ,q}_\infty$, for any $q$, as well as the homology $H^\star$ of the totalization of $\mathbb{K}^{\bullet ,\bullet}$ are finitely generated graded $ A[y_1,\ldots , y_c] $-modules.

On the other hand, since $\Hom_B(-, \mathbb{I}^q_{P})$ is an exact functor, if we start taking cohomology horizontally, then we obtain the first pages of the spectral sequence:
	\[
		^hE_1^{p,q} = \Hom_B \hspace{-0.07cm} \left( \Tor^A_p(M,N), \mathbb{I}^q_{P} \right), \quad ^{h}E_{2}^{p,q} = \
		\Ext_B^q \hspace{-0.07cm} \left( \Tor^A_p(M,N), P\right)
	\]
	and condition (iii) implies that there exists $p_0$ such that $^{h}E_{2}^{p,q}=0$ unless $q=r$ or $q=r-1$, if $p\geq p_0$.  Hence $^{h}E_{2}^{p,q}={^{h}{E_{\infty}^{p,q}}}$ for $p\ge p_0$.

	Taking direct sum over $ p \ge p_0 + r $ and using the naturality of Eisenbud operators, as in \ref{grad-mod-struc-Ext-Tor}, we obtain a short exact sequence of graded $ A[y_1,\ldots,y_c] $-modules:

		\begin{align*}
		0 \longrightarrow  \bigoplus_{p \ge p_0 +r }  \hspace{-0.1cm} \hspace{-0.1cm} \Ext_B^r \hspace{-0.07cm} \left( \Tor^A_{p-r}(M,N),P \right)\hspace{-0.1cm}& \longrightarrow \bigoplus_{p \ge p_0+r} \hspace{-0.1cm} H^p \\ 
		& \longrightarrow  \bigoplus_{p \ge p_0 +r } \Ext_B^{r-1} \hspace{-0.07cm} \left( \Tor^A_{p-r+1}(M,N),P \right)\hspace{-0.1cm}\longrightarrow 0.
	\end{align*}
	The middle term is a finitely generated graded $A[y_1,\dots,y_c]$-module, as $H^\star$ is so. Hence the assertion follows.
\end{proof}
\begin{remark} Notice that whenever $A$ is *local Cohen-Macaulay (equivalently $Q$), one may apply the same line of proof with $B=A$, $P=\omega_A$ and $N$ replaced by a high syzygy to assume that $N$ is maximal Cohen-Macaulay. In this particular (but important) situation, the vertical spectral sequence abuts on step 2.
\end{remark}

\begin{remark}\label{rmk:thm:main-gen-version}

	 Theorem~\ref{thm:main} with $B$ *local Cohen-Macaulay shows that if there exits an integer $ r \ge 1 $ such that $ r - 1 \le \depth\left( \Tor_i^A(M,N) \right) $ and $ \dim \left( \Tor_i^A(M,N)  \right) \le r $  for all $ i \gg 0 $, then for any $q$, 
	\[
		\Ext_B^q \hspace{-0.07cm} \left( \Tor^A_\star (M,N), \omega_B\right)
	\]
is a finitely generated graded $ A[y_1,\ldots , y_c] $-module.
\end{remark}

\begin{proof}[Proof of Theorem~\ref{cor:lin-reg-Tor-dim-1}]
Set $W_i:= \Tor_{2i}^A(M,N)$. We will show the linearity of $ \reg( W_i ) $ for all $ i \gg 0 $. The result for $ \Tor_{2i+1}^A(M,N) $ follows similarly.  We adopt the notations of the proof of Theorem~\ref{thm:main} after choosing a graded epimorphism  $B\rightarrow Q$ with $B$ equidimensional Cohen-Macaulay and $P=\omega_B$ in this statement, and choose $i_0$ such that $\dim(W_i)\le 1$ for all $i\ge i_0$. Notice that $\Ext^{\dim B-j}_B(W_i,\omega_B)=0$ for all $i\ge i_0$ and $j\neq 0,1$. 
Set 
\begin{equation*}
H^0_{[0]}(M):= \bigcup_{\substack{I\subset A_0,\\
\dim(A_0/I)=0}} H^0_{I}(M)
\end{equation*}
 as in \cite[Section~7]{CJR13}. Let  $D_i:= \Ext^{\dim B-1}_B(W_i,\omega_B)$,  $E_i:= H^0_{[0]}(D_i )$, $F_i$ be defined by the exact sequence
\begin{equation}\label{exact sequence proof 1}
	0 \longrightarrow \bigoplus_{ i \ge i_0 } E_i \longrightarrow \bigoplus_{i\ge i_0} \Ext^{\dim B-1}_B(W_i,\omega_B) \longrightarrow \bigoplus_{i\ge i_0} F_i \longrightarrow 0,
\end{equation}
and $G_i:= \Ext^{\dim B}_B(W_i,\omega_B)$. By Theorem~\ref{thm:main}, $\bigoplus_i G_i$ and $\bigoplus_i D_i$ are finitely generated graded $A[y_1,\dots,y_c]$-modules. Hence, by \ref{exact sequence proof 1}, so are $\bigoplus_{i} E_i $ and $\bigoplus_{i} F_i$. Then Proposition~\ref{prop:BCH} shows that there exist $a,a',a'' \in \{ \deg (f_j) : 1 \le j \le c \}$, $e,e',e''\in \mathbb{Z}\cup \lbrace -\infty\rbrace$ and $i_0'\ge i_0$, such that for all $i\ge i_0'$,
\begin{align*}
\indeg(G_i)&= -ai-e,\\
\indeg( F_i)&= -a'i-e',\\
\indeg( E_i)&= -a''i-e''.
\end{align*}
We will now show that for $i\ge i_0'$,
$$
\reg(W_i)=r(i):= \max\left\{ ai+e, a'i+e', a''i+e''+1\right\}.
$$
In view of \cite[Lemma~7.2]{CJR13}, for any graded $A$-module $M$, $H^0_{[0]}(M)_{\mu}=H^0_{[0]}(M_{\mu})$, and for $\mathfrak{m}_0\in \Specmax(A_0)$,
\begin{equation}\label{CJR equation}
H^0_{[0]}(M)\otimes_{A_0}(A_0)_{\mathfrak{m}_0} = H^0_{\mathfrak{m}_0}\hspace*{-0.05cm}\big(M \otimes_{A_0}(A_0)_{\mathfrak{m}_0}\big).
\end{equation}
Recall that $\reg(W_i)=\max \left\{ \reg\big(W_i\otimes_{A_0}(A_0)_{\mathfrak{m}_0}\big) : \mathfrak{m}_0\in \Specmax(A_0)\right\}$. Let $\mathfrak{m}_0\in \Specmax(A_0)$, $\mathfrak{m}:= \mathfrak{m}_0+A_{+}$ and write 
$$
-':=-\otimes_{A_0}(A_0)_{\mathfrak{m}_0} \quad \text{and} \quad -^{\vee}:= {^*\Hom}_{A_0'}\hspace*{-0.05cm}\big( - , E_{A_0'}(A_0'/\mathfrak{m}'_0)\big).
$$
 Applying $-'$  to the  sequence \ref{exact sequence proof 1}, we get by \ref{CJR equation} for $i\ge i_0$ the exact sequences 
\begin{equation}\label{exact sequence 2}
	0 \longrightarrow  H^0_{\mathfrak{m}_0} (D_i') \longrightarrow  D_i' \longrightarrow  F_i' \longrightarrow 0.
\end{equation}
Note that $(D_i')^{\vee}\cong H^1_{\mathfrak{m}}(W_i')$ by \cite[Cor.~3.5.9]{BH98}. With the notations as in  Lemma~\ref{lem:reg_and_regm}, and considering the composed functor spectral sequence
\[
	H^p_{\mathfrak{m}_0}\hspace*{-0.05cm}\big(H^q_{A'_+}(-)\big) \Longrightarrow H^{p+q}_{\mathfrak{m}}(-)
\]
as in the proof of \cite[Prop. 3.4]{Has12}, for $i\ge i_0$, we have  the following exact sequences of graded $A'$-modules:
\begin{equation}\label{exact sequence 3}
0\longrightarrow  H^1_{\mathfrak{m}_0}\hspace*{-0.05cm}\big(H^0_{A'_+}(W_i')\big) \longrightarrow  H^{1}_{\mathfrak{m}}(W_i')\longrightarrow  H^0_{\mathfrak{m}_0}\hspace*{-0.05cm}\big( H^1_{A'_+}(W_i')\big)\longrightarrow 0.
\end{equation}
Since $H^0_{A'_+}(W_i')_{\mu}$ is a finitely generated $A_0$-module of dimension at most $1$ for any $\mu$,  $H^1_{\mathfrak{m}_0}\hspace*{-0.05cm}\big(H^0_{A'_+}(W_i')\big)^{\vee}= \bigoplus_\mu H^1_{\mathfrak{m}_0}\hspace*{-0.05cm}\big(H^0_{A'_+}(W_i')_{\mu}\big)^{\vee}$ has no $\mathfrak{m}_0$-torsion,  it follows that 
\[
	H^0_{\mathfrak{m}_0}\hspace*{-0.05cm}\big(H^1_{A'_+}(W_i')\big)^\vee \cong H^0_{\mathfrak{m}_0}(D'_i).
\]
It shows that \ref{exact sequence 3} is the Matlis dual of \ref{exact sequence 2}, and the Matlis dual of $G'_i$ is $H^0_{\mathfrak{m}}(W_i')$.  In particular, we get
\begin{align*}
\eend \hspace*{-0.05cm}\big(H^0_{\mathfrak{m}}(W_i')\big) &= -\indeg(G'_i), \\
\eend \left(H^1_{\mathfrak{m}_0}\hspace*{-0.05cm}\big(H^0_{A'_+}(W_i')\big)\right) &= -\indeg(F'_i),\\
\eend \left(H^0_{\mathfrak{m}_0}\hspace*{-0.05cm}\big(H^1_{A'_+}(W_i')\big)\right)&= -\indeg(E'_i).
\end{align*}
As for any graded $A$-module $M$, $\indeg(M')\ge \indeg(M)$ with equality for some $\mathfrak{m}_0\in\Specmax(A_0)$ if $\indeg(M)\neq -\infty$, it follows from Lemma~\ref{lem:reg_and_regm} that $\reg(W_i)=r(i)$ for all $i\ge i'_0$.
\end{proof}

\begin{proposition}\label{prop:zero-two-loc-coh}
In Theorem \ref{thm:main}, assume $B$ is *local Cohen-Macaulay, $P=\omega_B$ and replace the hypothesis {\rm (iii)} by the following weaker assumption 

{\rm (iii)'} $ \dim \left( \Tor_i^A(M,N) \right) \le 2 $ for all $i\gg 0$. 

Then, for any $q\not= \dim B,\dim B-2$,  $\Ext_B^q \hspace{-0.07cm} \left( \Tor^A_\star (M,N), \omega_B \right) $
is a finitely generated graded $ A[y_1,\ldots , y_c] $-module, and the following are equivalent :\\
{\rm (a)} $\Ext_B^{\dim B} \hspace{-0.07cm} \left( \Tor^A_\star (M,N), \omega_B\right)$ is a finitely generated graded $ A[y_1,\ldots , y_c] $-module,\\
{\rm (b)} $\Ext_B^{\dim B-2} \hspace{-0.07cm} \left( \Tor^A_\star (M,N), \omega_B\right)$ is a finitely generated graded $ A[y_1,\ldots , y_c] $-module.
\end{proposition}

\begin{proof}
Using the same argument as in the proof of Theorem~\ref{thm:main}, the abutment of the spectral sequence is obtained in the third page for the following components:
	\begin{equation}\label{Spectral-seq-page3}
		^{h}E_{\infty}^{p,q} = \, ^{h}E_{3}^{p,q} =
		\begin{cases}
			\Coker(\Phi_{p+1}) & \mbox{if } p \ge p_0 - 2 \mbox{ and } q = b,\\
			\Ext_B^q \hspace{-0.07cm} \left( \Tor^A_p(M,N), \omega_B\right) & \mbox{if } p \ge p_0 - 1 \mbox{ and } q=b-1, \\
			\Ker(\Phi_p)&  \mbox {if } p\ge p_0 -1\mbox{ and } q = b-2,\\
			0 									& \mbox {if } p\ge p_0 \mbox{ and } q \notin \{ b, b-1, b-2 \},
		\end{cases}
	\end{equation}
	where $ \Phi_p: \Ext_B^{b-2} \hspace{-0.07cm} \left( \Tor^A_p(M,N), \omega_B\right) \longrightarrow \Ext_B^b \hspace{-0.07cm} \left( \Tor^A_{p-1}(M,N), \omega_B\right) $ are the induced maps in the second page of the spectral sequence. For every $ q $, the graded $ A[y_1,\dots , y_c] $-module $ \bigoplus_{p} {}^{h}E_{\infty}^{p,q} $ is finitely generated, because the spectral sequence identifies it as a quotient of two graded submodules of $ H^\star $. Thus, according to \eqref{Spectral-seq-page3}, it shows that
	\begin{equation}\label{abutment-fg}
		\bigoplus_{p \ge p_0 - 1} \hspace{-0.2cm} \Coker(\Phi_p), \quad \bigoplus_{p \ge p_0 - 1}\Ext_B^{b-1} \hspace{-0.07cm} \left( \Tor^A_p(M,N), \omega_B\right)\quad \mbox{and }\; \bigoplus_{p \ge p_0-1 } \hspace{-0.2cm} \Ker(\Phi_p)
	\end{equation}
	are finitely generated over $ A[y_1,\dots , y_c] $. For completing the proof, we use \eqref{abutment-fg} and the exact sequence
	\begin{align*}
		0 \longrightarrow \bigoplus_{p} \Ker(\Phi_p) & \longrightarrow  \bigoplus_{p} \Ext_B^{b-2} \hspace{-0.07cm} \left( \Tor^A_p(M,N), \omega_B\right) \\
		& \longrightarrow  \bigoplus_{p}\Ext_B^b \hspace{-0.07cm} \left( \Tor^A_{p-1}(M,N), \omega_B\right) \longrightarrow \bigoplus_{p} \Coker(\Phi_p) \longrightarrow 0
	\end{align*}
	of graded modules over $ A[y_1,\dots , y_c] $.	
\end{proof}
\begin{remark}
Whenever $B$ is a standard graded Gorenstein ring over a field, and $W$ or $P$ has finite projective dimension over $B$, the regularity of $W$ is provided by the formula \cite[3.2]{CD08}:
$$
\reg (W)=\reg (B)+\indeg (P)-\min_j \left\{ \indeg \left( \Ext^j_B(W,P) \right) + j \right\}.
$$
Hence Theorem \ref{thm:main} offers other choices of $P$ that could be used to deduce the linearity of the regularity for high Tor modules in specific situations, or to derive its value. To emphasize this remark, we recall now what Theorem \ref{thm:main} and this fact says whenever $Q$ is a polynomial ring over a field.
\end{remark}

\begin{proposition}\label{cor:thmain}
	Let  $Q$ be a polynomial ring over a field, $A := Q/({\bf f})$, where $ {\bf f}  := f_1,\ldots,f_c $ is a homogeneous $ Q $-regular sequence. Let $M$, $N$ and $P$ be finitely generated graded $A$-modules and $r\in \mathbb{N}$. If
	$$
	\Ext_Q^q \hspace{-0.05cm}\big(\Tor^A_i (M,N), P\big)=0, \forall \, i \gg 0 \ {\rm if}\ q\not\in \{ r-1,r\} ,
	$$
then

{\rm (i)} 	$\Ext_Q^j \hspace{-0.05cm} \left( \Tor^A_\star (M,N), P\right)$ is a finitely generated graded $ A[y_1,\ldots , y_c] $-module, for any $j$.

{\rm (ii)} $\reg \hspace{-0.05cm} \left(\Tor^A_i (M,N)\right)=\indeg (P) - \min_j \left\{ \indeg \left(\Ext^j_Q\hspace{-0.05cm} \left( \Tor^A_i (M,N),P \right) \right) + j \right\}$, for any $i$.
\end{proposition}

\begin{remark}
When $Q$ is *local,	along the same lines as in the  proof  of Theorem~\ref{cor:lin-reg-Tor-dim-1}, Remark~\ref{rmk:thm:main-gen-version} yields the following. With Hypothesis~\ref{hyp:1}, further assume that $ r - 1 \le \depth\left( \Tor_i^A(M,N) \right) $ and $ \dim \left( \Tor_i^A(M,N) \right) \le r $  for all $ i \gg 0 $, where $ r \ge 1 $ is an integer. Then, for every $ l \in \{0,1\} $, there exist $ a_l \in \{ w_j : 1 \le j \le c \} $ and $ e_l \in \mathbb{Z} \cup \{ - \infty \} $ such that $ \reg\left( \Tor_{2i+l}^A(M,N) \right) = a_l \cdot i + e_l $ for all $ i \gg 0 $.
\end{remark}


\section{Examples on linearity of regularity}\label{sec:examples:linearity}

 Here we construct an example, which shows that the result in Theorem~\ref{cor:lin-reg-Tor-dim-1} does not necessarily hold true for higher dimension. In this example, though $ \reg\left( \Tor_i^A(M,N) \right) $ is asymptotically linear in $ i $, but unlike Ext modules, the leading term of the linear function for Tor depends on the modules $ M $ and $ N $.
 
 \begin{example}\label{exam:reg-Tor-Ext-setup1}
 	Let $ Q := K[Y,Z,V,W] $ be a polynomial ring with usual grading over a field $ K $, and $ A := Q/(Y^2,Z^2) $. Write $ A = K[y,z,v,w] $, where $ y, z, v $ and $ w $ are the residue classes of $ Y,Z,V $ and $ W $ respectively. Fix an integer $ m \ge 1 $. Set
 	\[
 	M := \Coker \left( \begin{bmatrix}
 	y & z & 0 & 0 \\
 	-v^m & -w^m & y & z
 	\end{bmatrix} : {\begin{array}{c}
 		A(-m)^2\\
 		\bigoplus \\
 		A(-1)^2
 		\end{array}} \longrightarrow {\begin{array}{c}
 		A(-m+1)\\
 		\bigoplus \\
 		A
 		\end{array}} \right)
 	\]
 	and $ N := A/(y,z) $. Then, for every $ i \ge 1 $, we have
 	\begin{enumerate}[{\rm (i)}]
 		\item $ \indeg\left( \Ext_A^i(M,N)  \right) = -i - m+1  $ and $ \reg\left( \Ext_A^i(M,N)  \right) = - i $.
 		\item $ \indeg\left( \Tor_i^A(M,N)  \right) = i $ and $ \reg\left( \Tor_i^A(M,N)  \right) = (m+1)i + (2m-2) $.		
 	\end{enumerate}
 \end{example}
	
	We postpone the proof of Example~\ref{exam:reg-Tor-Ext-setup1} until the end of this section.

	\begin{remark}
		In Example~\ref{exam:reg-Tor-Ext-setup1}(ii), though $ \reg\left( \Tor_i^A(M,N)  \right) $ is linear in $ i $, but the leading term is $ (m+1) $, which can be as large as possible depending on $ M $. In particular, it shows that the result in Theorem~\ref{cor:lin-reg-Tor-dim-1} is not necessarily true for higher dimension of $ \Tor_i^A(M,N) $. In the proof of Example~\ref{exam:reg-Tor-Ext-setup1}(ii), since $ \dim(\Ker(\Phi_i)) = 2 $, it follows that $ \dim\left( \Tor_i^A(M,N)  \right) = 2 $ for all $ i \ge 1 $.
	\end{remark}
	
	\begin{remark}
		In view of Theorem~\ref{thm:Chardin} and Example~\ref{exam:reg-Tor-Ext-setup1}(ii), by comparing the coefficients of $ i $ from both sides, we can conclude that the inequalities in Theorem~\ref{thm:Chardin} do not necessarily hold true for higher dimension of Tor modules.
	\end{remark}
	
	\begin{remark}\label{rmk:fg-tor}
		With Setup~\ref{setup1}, the graded modules
		\begin{equation*}
		\bigoplus_{i \ge 0} H_{A_+}^0 \hspace{-0.1cm}\left( \Tor_i^A(M,N)\right)\hspace{-0.1cm}^{\vee} \quad \text{and} \quad \bigoplus_{i \ge 0} H_{A_+}^2 \hspace{-0.1cm}\left( \Tor_i^A(M,N)\right)\hspace{-0.1cm}^{\vee}
		\end{equation*}
		are not finitely generated over $ A[y_1,\dots , y_c] $. Otherwise, using Proposition~\ref{prop:zero-two-loc-coh}, as in Theorem~\ref{cor:lin-reg-Tor-dim-1}, one obtains that $ \reg\left( \Tor_{2i}^A(M,N)  \right) $ is linear in $ i $ with leading coefficient $ 2 $, which is a contradiction because $ \reg\left( \Tor_{2i}^A(M,N)  \right) = 2(m+1)i + (2m-2) $.
	\end{remark}
	
\begin{setup}\label{setup1}
	Along with the hypotheses of Example~\ref{exam:reg-Tor-Ext-setup1}, for every integer $ n \ge 1 $, we set the matrices $ B_{2n} $ and $ C_{2n} $ of order $ 2n \times (2n+1) $ as follows:
	\[ 
	B_{2n} := \hspace{-0.1cm} \begin{bmatrix}
	y          & -z   &  0        & 0	&  \cdots   & 0 \\
	0          & y   & z       & 0		& \cdots   & 0 \\
	0          & 0   &  y        & -z    & \cdots   & 0 \\
	\vdots  & \vdots  & \vdots & \ddots  & \ddots  & \vdots \\
	0        &   0   &  0		 &    \cdots	& y     & z \\
	\end{bmatrix}\hspace{-0.1cm},\;
	C_{2n} := \hspace{-0.1cm}  \begin{bmatrix}
	v^m      & w^m   &  0        & 0	&  \cdots   & 0 \\
	0          & -v^m   & w^m       & 0		& \cdots   & 0 \\
	0          & 0   &  v^m        & w^m    & \cdots   & 0 \\
	\vdots  & \vdots  & \vdots & \ddots  & \ddots & \vdots \\
	0        &   0   &  0		 &    \cdots	& -v^m     & w^m \\
	\end{bmatrix}
	\]
	while for $ n \ge 0 $, we set the matrices $ B_{2n+1} $ and $ C_{2n+1} $ of order $ (2n+1) \times (2n+2) $ as follows:
	\[ 
	B_{2n+1} \hspace{-0.15cm} := \hspace{-0.2cm} \begin{bmatrix}
	y          & z   &  0        & 0	&  \cdots   & 0 \\
	0          & y   & -z       & 0		& \cdots   & 0 \\
	0          & 0   &  y        & z    & \cdots   & 0 \\
	\vdots  & \vdots  & \vdots & \ddots  & \ddots & \vdots \\
	0        &   0   &  0		 &    \cdots	& y     & z \\
	\end{bmatrix}\hspace{-0.1cm},\,
	C_{2n+1} \hspace{-0.15cm} := \hspace{-0.2cm} \begin{bmatrix}
	-v^m      \hspace{-0.2cm}& -w^m   \hspace{-0.2cm}&  0        \hspace{-0.2cm}& 0	\hspace{-0.2cm} &  \cdots   \hspace{-0.2cm}& 0 \\
	0          \hspace{-0.2cm}& v^m   \hspace{-0.2cm}& -w^m       \hspace{-0.2cm}& 0		\hspace{-0.2cm} & \cdots   \hspace{-0.2cm}& 0 \\
	0          \hspace{-0.2cm}& 0   \hspace{-0.2cm}&  -v^m        \hspace{-0.2cm}& -w^m    \hspace{-0.2cm} & \cdots   \hspace{-0.2cm}& 0 \\
	\vdots  \hspace{-0.2cm}& \vdots  \hspace{-0.2cm}& \vdots \hspace{-0.2cm}& \ddots  \hspace{-0.2cm}& \ddots \hspace{-0.2cm}& \vdots \\
	0        \hspace{-0.2cm}&   0   \hspace{-0.2cm}&  0		 \hspace{-0.2cm}&    \cdots	\hspace{-0.2cm} & -v^m     \hspace{-0.2cm}& -w^m \\
	\end{bmatrix}
	\]
	Note that $ B_n $ and $ C_n $ are matrices over $ A $ both of order $ n \times (n+1) $ for every $ n \ge 1 $. Finally, we set a block matrix $ D_n $ of order $ 2n \times (2n+2) $ as follows:
	\[
	D_n := \begin{bmatrix}
	B_n & O_n \\
	C_n & B_n
	\end{bmatrix} \quad \mbox{for every } n \ge 1,
	\]
	where $ O_n $ denotes the matrix of order $ n \times (n+1) $ with all entries $ 0 $.
\end{setup}

The following relations of $ B_n $ and $ C_n $ $ (n \ge 1) $ help us to build minimal free resolution of $ M $.

\begin{proposition}\label{prop:BC+CB=0}
	With Setup~\ref{setup1}, for every $ n \ge 1 $, $ B_n C_{n+1} + C_n B_{n+1} = 0 $.
\end{proposition}

\begin{proof}
	We use induction on $ n $. It can be verified that $ B_1 C_2 + C_1 B_2 = 0 $ and $ B_2 C_3 + C_2 B_3 = 0 $. Assuming the equality $ B_p C_{p+1} + C_p B_{p+1} = 0 $ for $ p \le n $, we verify it for $ n+1 $. We may assume that $ n $ is even, say $ 2q $. The case when $ n $ is odd can be treated in a similar way. Note that
	\begin{align*}
	B_{2q+1} C_{2q+2} & = \hspace{-0.15cm}\left[\begin{array}{cc|ccccc}
	y v^m  & y w^m - z v^m & z w^m   			&  0       				& 0 		& \cdots   		& 0 \\
	0         & -y v^m         & y w^m -z v^m   	&  -z w^m   		& 0			&  \cdots   	& 0 \\ \hline
	0         & 0          		&   						&        					& 			&					&   \\
	\vdots  & \vdots 		&    						&  B_{2q-1} C_{2q} &    		&    				&   \\
	0 		& 0  				&  							& 								&  				&  					&   
	\end{array}	\right] \mbox{and}\\
	C_{2q+1} B_{2q+2} & = \hspace{-0.15cm}\left[\begin{array}{cc|ccccc}
	-y v^m  & z v^m - y w^m & -z w^m   			&  0       				& 0 		& \cdots   		& 0 \\
	0         & y v^m         & z v^m - y w^m	&  z w^m   		& 0			&  \cdots   	& 0 \\ \hline
	0         & 0          		&   						&        					& 			&					&   \\
	\vdots  & \vdots 		&    						&  C_{2q-1} B_{2q} &    		&    				&   \\
	0 		& 0  				&  							& 								&  				&  					&   
	\end{array}	\right].
	\end{align*}
	Hence induction hypothesis yields that $ B_{2q+1} C_{2q+2} + C_{2q+1} B_{2q+2} = 0 $.
\end{proof}

Here we construct graded minimal free resolutions of $ M $ and $ N $ over $ A $.

\begin{lemma}\label{lem:reso-setup1}
	With Setup~\ref{setup1}, the following statements hold true.\\
	{\rm (i)} A graded minimal free resolution of $ N $ over $ A $ is given by $ \mathbb{F}_{\bullet}^N : $
	\[
	\cdots \longrightarrow  A(-n)^{n+1} \xrightarrow{B_n} A(-n+1)^n \longrightarrow\cdots \xrightarrow{B_2} A(-1)^2 \stackrel{B_1}{\longrightarrow} A\longrightarrow 0
	\]
	{\rm (ii)} A graded minimal free resolution of $ M $ over $ A $ is given by $ \mathbb{F}_{\bullet}^M : $
\[
\cdots
\xrightarrow{D_{n+1}}
\hspace{-0.3cm}
\begin{array}{c}
						A(-m-n+1)^{n+1}\\
						\bigoplus\\
						A(-n)^{n+1}
						\end{array}
						\hspace{-0.2cm}
	\xrightarrow{\; D_n \;}		
\hspace{-0.2cm}
\begin{array}{c}
			A(-m-n+2)^n\\
			\bigoplus\\
			A(-n+1)^n
			\end{array}
			\hspace{-0.3cm}
			\longrightarrow
			\cdots
			\xrightarrow{\; D_1 \;}
			\hspace{-0.2cm}
			\begin{array}{c}
A(-m+1)\\
\bigoplus\\
A
\end{array}
\hspace{-0.3cm}
\longrightarrow
0
\]
\end{lemma}

\begin{proof}
	(i) Set $ N_1 := A/(y) $ and $ N_2 := A/(z) $. Clearly,
	\begin{align*}
		&\mathbb{F}_{\bullet}^{N_1} : \;\; \cdots \xrightarrow{\; y \;} A(-2) \xrightarrow{\; y \;} A(-1) \xrightarrow{\; y \;} A \rightarrow 0		 \quad \mbox{and}\\
		&\mathbb{F}_{\bullet}^{N_2} : \;\; \cdots \xrightarrow{\; z \;} A(-2) \xrightarrow{\; z \;} A(-1) \xrightarrow{\; z \;} A \rightarrow 0
	\end{align*}
	are graded minimal $ A $-free resolutions of $ N_1 $ and $ N_2 $ respectively. Since $ \mathbb{F}_{\bullet}^{N_1} \otimes_A N_2 $ is acyclic, it follows that $ \Tor_i^A(N_1,N_2) = 0 $ for all $ i \ge 1 $. Let $ \mathbb{F}_{\bullet} $ be the tensor product of $ \mathbb{F}_{\bullet}^{N_1} $ and $ \mathbb{F}_{\bullet}^{N_2} $ over $ A $; see \cite[pp~614]{Rot09}. Note that the homology $ H_i(\mathbb{F}_{\bullet}) = \Tor_i^A(N_1,N_2) $ (cf. \cite[10.22]{Rot09}). Thus, since $ H_i(\mathbb{F}_{\bullet}) = 0 $ for all $ i \ge 1 $, $ \mathbb{F}_{\bullet} $ provides a free resolution of $ N_1 \otimes_A N_2 = A/(y,z) = N $. It follows from the definition of tensor product of complexes that $ \mathbb{F}_{\bullet} $ is same as the desired free resolution $ \mathbb{F}_{\bullet}^N $.
	
	(ii) Set $ \mathbb{G} := \mathbb{F}_{\bullet}^N $, the resolution shown in (i), and $ \mathbb{H} := \mathbb{G}[1](-m+1) $, i.e.,
	\[
		\mathbb{H}_n = \mathbb{G}_{n+1}(-m+1) \quad\mbox{and}\quad d_n^{\mathbb{H}} = (-1) d_{n+1}^{\mathbb{G}} \quad \mbox{for every } n;
	\]
	see \cite[1.2.8]{Wei94}. We construct a map $ f : \mathbb{H} \to \mathbb{G} $ as follows: the $ n $th component $ f_n : \mathbb{H}_n \to \mathbb{G}_n $ of $ f $ is defined by $ (-1) C_{n+1} $.	By virtue of Proposition~\ref{prop:BC+CB=0}, $ f $ is a homogeneous map of chain complexes. We consider the mapping cone $ \Cone(f) $; see \cite[1.5.1]{Wei94} for its definition. Note that $ \Cone(f)_n = \mathbb{H}_{n-1} \oplus \mathbb{G}_n $ with the $ n $th differential
	\[
		\begin{bmatrix}
		- d_{n-1}^{\mathbb{H}} & 0 \\
		- f_{n-1} & d_n^{\mathbb{G}}
		\end{bmatrix} : {\begin{array}{lcl}
			\mathbb{H}_{n-1} & \longrightarrow & \mathbb{H}_{n-2} \\
			\bigoplus & \searrow & \bigoplus \\
			\mathbb{G}_n & \longrightarrow & \mathbb{G}_{n-1}
			\end{array}}
	\]
	which is nothing but $ D_n $ as given in the desired resolution. Since $ H_n(\mathbb{G}) = 0 = H_{n-1}(\mathbb{H}) $ for every $ n \ge 1 $, in view of \cite[1.5.2]{Wei94}, we have $ H_n(\Cone(f)) = 0 $ for every $ n \ge 1 $. Hence $ \Cone(f) $ provides the desired free resolution $ \mathbb{F}_{\bullet}^M $.
\end{proof}

\begin{remark}
Part (i) in this Lemma could also be deduced from \cite[6.1.8]{Av98} and tools from \cite{EP16} could likely help to improve and shorten our elementary arguments for (ii). This also applies to Lemma \ref{lem:reso-setup2}.
\end{remark}

\begin{para}[Computations of $ \Tor_i^A(M,N) $ and $ \Ext_A^i(M,N) $ with Setup~\ref{setup1}]\label{para:example1-tor-ext}
	In view of 
	Lemma~\ref{lem:reso-setup1}(ii),  we obtain that the complex $ \mathbb{F}_{\bullet}^M \otimes_A N $ is given by
	\[
	\cdots
	\xrightarrow{D'_{n+1}}
	\hspace{-0.3cm}
	\begin{array}{c}
	N(-m-n+1)^{n+1}\\
	\bigoplus\\
	N(-n)^{n+1}
	\end{array}
\hspace{-0.2cm}	 \xrightarrow{\; D'_n \;} \hspace{-0.2cm}
\begin{array}{c}
	N(-m-n+2)^n\\
	\bigoplus\\
	N(-n+1)^n
	\end{array}
	\hspace{-0.3cm}
	\longrightarrow
	\cdots
	\xrightarrow{\; D'_1 \;} \hspace{-0.2cm} 
	\begin{array}{c}
	N(-m+1)\\
	\bigoplus\\
	N
	\end{array}
	\hspace{-0.3cm}
	\longrightarrow
	0
	\]
	where
	\[
		D'_n := D_n \otimes_A A/(y,z) = \begin{bmatrix}
			O_n & O_n \\
			C_n & O_n
		\end{bmatrix} \quad \mbox{for every } n \ge 1.
	\]
	This yields that
	\begin{equation}\label{Tor-n-setup1}
		\Tor_n^A(M,N) = \begin{array}{l}
		\Ker\Big( N(-m-n+1)^{n+1} \xrightarrow{\; C_n \;}  N(-n+1)^n  \Big) \\
		\;\; \bigoplus \\
		\Coker\Big(  N(-m-n)^{n+2} \xrightarrow{\; C_{n+1} \;}  N(-n)^{n+1} \Big)
		\end{array} \mbox{ for } n \ge 1.
	\end{equation}
	It follows that $ \Tor_n^A(M,N)_{\mu} = 0 $ for every $ \mu < n $, and $ \Tor_n^A(M,N)_n \neq 0 $. Therefore
	\begin{equation}\label{eqn:indeg-tor}
		\indeg\left( \Tor_n^A(M,N)  \right) = n \quad \mbox{for every } \; n \ge 1.
	\end{equation}
	To compute Ext modules, consider the complex $ \Hom_A(\mathbb{F}_{\bullet}^M,N) $, which is given by
	\[
	0 \longrightarrow \hspace{-0.3cm} \begin{array}{c}
	N(m-1)\\
	\bigoplus\\
	N
	\end{array}
	\hspace{-0.2cm} \xrightarrow{ (D'_1)^t } \cdots \longrightarrow \hspace{-0.3cm}
	\begin{array}{c}
	N(m+n-2)^n\\
	\bigoplus\\
	N(n-1)^n
	\end{array}
	\hspace{-0.2cm} \xrightarrow{ (D'_n)^t } \hspace{-0.2cm}
	\begin{array}{c}
	N(m+n-1)^{n+1}\\
	\bigoplus\\
	N(n)^{n+1}
	\end{array}
	\hspace{-0.3cm} \xrightarrow{(D'_{n+1})^t} \cdots
	\]
	where $ (-)^t $ stands for the transpose of a matrix. Hence it can be observed that
	\begin{equation}\label{Ext-n-setup1}
		\Ext^n_A(M,N) = \begin{array}{l}
	\Coker\Big( N(n-1)^n \xrightarrow{\;\; C^t_n \;\;} N(m+n-1)^{n+1} \Big) \\
	\;\; \bigoplus \\
	\; \Ker\Big( N(n)^{n+1} \xrightarrow{\;\; C^t_{n+1} \;\;} N(m+n)^{n+2} \Big)
	\end{array} \mbox{ for } n \ge 1.
	\end{equation}
\end{para}

We are now able to provide a proof for the example.

\begin{proof}[Proof of Example~\ref{exam:reg-Tor-Ext-setup1}]
	In view of \eqref{Tor-n-setup1} and \eqref{Ext-n-setup1}, it suffices to study the regularity of  kernel and cokernel of the following maps :
	\[
		\Phi_n \hspace{-0.05cm}:= \hspace{-0.05cm} N(-m-n+1)^{n+1} \hspace{-0.05cm} \xrightarrow{ C_n } \hspace{-0.05cm} N(-n+1)^n\hspace{-0.05cm},\; \Psi_n \hspace{-0.05cm} := \hspace{-0.05cm} N(n-1)^n \hspace{-0.05cm} \xrightarrow{ C_n^t } \hspace{-0.05cm} N(m+n-1)^{n+1}
	\]
	for all $ n \ge 1 $. Since $ N $ is annihilated by $ (y,z) $, we can substitute $ N $ with $ R:= K[V,W] $, and $ v^m, w^m $ in the entries of the matrices $ C_n $ with $ V^m, W^m $ respectively.
	
	(i) Since $I_n(C^t_n) $, the ideal of maximal minors of $C_n^t $, has depth $ = 2 $, by the Hilbert-Burch Theorem (cf. \cite[Thm. 1.4.17]{BH98}), we have a graded minimal $ R $-free resolution of $ \Coker(\Psi_n) $ :
	\begin{equation}\label{eqn:Hilbert-Burch}
		0 \longrightarrow R(n-1)^n \xrightarrow{\; C^t_n \;} R(m+n-1)^{n+1} \stackrel{\pi}{\longrightarrow} \Coker(\Psi_n) = I_n(C^t_n) \longrightarrow 0,
	\end{equation}
	where $ \pi $ sends the standard basis element $ e_i $ to $ (-1)^i \delta_i $, and $ \delta_i $ denotes the $ n \times n $ minor of $ C_n^t $ with the $ i $th row deleted for $ 1 \le i \le n+1 $. Therefore, for every $ n \ge 1 $, one obtains that $ \Ker(\Psi_n) = 0 $, $ \indeg(\Coker(\Psi_n)) = -m-n+1 $ and $ \reg(\Coker(\Psi_n)) = -n $. Thus it follows from \eqref{Ext-n-setup1} that for every $ n \ge 1 $,
	\begin{align*}
		\indeg(\Ext^n_A(M,N)) &= \min\{ \indeg(\Coker(\Psi_n)), \indeg(\Ker(\Psi_{n+1})) \} \\
		& =  -n-m+1 \quad \mbox{and}\\
		\reg(\Ext^n_A(M,N)) & = \hspace{-0.05cm}\max\{ \reg(\Coker(\Psi_n)), \reg(\Ker(\Psi_{n+1})) \} = -n.
	\end{align*}
	
	(ii) By \eqref{eqn:Hilbert-Burch}, since $ \deg(\delta_i) = mn $, we get an exact sequence of graded $ R $-modules:
	\[
		0 \longrightarrow R(n-1)^n \xrightarrow{ C^t_n } R(m+n-1)^{n+1} \xrightarrow{E_n := \left[-\delta_1 \; \delta_2 \;\cdots \; (-1)^{n+1}\delta_{n+1}\right]} R(mn+m+n-1).
	\]
	Applying $ \Hom_{R}(-,R) $, we obtain a complex
	\begin{equation}\label{eqn:BE}
		0 \longrightarrow R(-mn-m-n+1) \xrightarrow{\; E_n^t \;} R(-m-n+1)^{n+1} \xrightarrow{\; C_n \;} R(-n+1)^n \longrightarrow 0
	\end{equation}
	which is acyclic, due to Buchsbaum-Eisenbud acyclicity criterion \cite[Thm. 1.4.13]{BH98}.	Thus \eqref{eqn:BE} is a graded minimal $ R $-free resolution of $ \Coker(\Phi_n) $, and $ \Ker(\Phi_n) \cong R(-mn-m-n+1) $.

	Hence it follows from \eqref{Tor-n-setup1} that for every $ n \ge 1 $,
	\begin{align}\label{eqn:reg-tor}
		\reg\left(\Tor_n^A(M,N)\right) & = \max \lbrace \reg(\Ker(\Phi_n)), \reg(\Coker(\Phi_{n+1})) \rbrace \\
		& = \max\{ (m+1)n+m-1, (m+1)(n+1)+m-3 \} \nonumber\\
		& =	(m+1)n+(2m-2). \nonumber
	\end{align}

	Thus \eqref{eqn:indeg-tor} and \eqref{eqn:reg-tor} yield the assertion (ii).
\end{proof}

\section{Examples on nonlinearity of regularity}\label{sec:examples:nonlinearity}

The aim of this section is to show that $ \reg\left( \Tor_{2i}^A(M,N) \right) $ and $ \reg\left( \Tor_{2i+1}^A(M,N) \right) $ need not be asymptotically linear in $ i $ even over a complete intersection ring $ A $. We give the following example over a codimension three complete intersection ring in positive characteristic.

\begin{example}\label{exam:reg-Tor-Ext-setup2}
	Let $ Q := K[X,Y,Z,U,V,W] $ be a standard graded polynomial ring over a field $ K $ of characteristic $ 2 $, and $ A := Q/(X^2,Y^2,Z^2) $. We write $ A = K[x,y,z,u,v,w] $, where $ x,y,z,u,v $ and $ w $ are the residue classes of $ X,Y,Z,U,V $ and $ W $ respectively. Set
	\[
	M := \Coker \left( \begin{bmatrix}
	x & y & z & 0 & 0 & 0 \\
	u & v & w & x & y & z 
	\end{bmatrix} : A(-1)^6 \longrightarrow A^2 \right)
	\quad \mbox{and} \quad N:= A/(x,y,z).
	\]
	Then, for every $ n \ge 1 $, we have
	\begin{enumerate}[{\rm (i)}]
		\item $ \indeg\left( \Ext_A^n(M,N)  \right) = -n $ and $ \reg\left( \Ext_A^n(M,N)  \right) = - n $.
		\item $ \indeg\left( \Tor_n^A(M,N)  \right) = n $ and $ \reg\left( \Tor_n^A(M,N)  \right) = n + f(n) $, where
		\[
		f(n) := \left\{ \begin{array}{ll}
		2^{l+1} - 2 & \mbox{if } n = 2^l - 1 \\
		2^{l+1} - 1 & \mbox{if } 2^l \le n \le 2^{l+1} - 2
		\end{array}\right. \quad \mbox{for all integers $ l \ge 1 $}.
		\]
	\end{enumerate}
\end{example}

\begin{remark}
	Example~\ref{exam:reg-Tor-Ext-setup2}(ii) shows that $ \reg( \Tor_{2n}^A(M,N)) $ and $ \reg( \Tor_{2n+1}^A(M,N) ) $ are not asymptotically linear as functions of $ n $. 
	Moreover, one obtains that $ n + 1 \le f(n) \le 2n $ for every $ n \ge 1 $, while $ f(n) = n + 1 $ if $ n = 2^{l+1} - 2 $, and $ f(n) = 2n $ if $ n = 2^l - 1 $ for $ l \ge 1 $. Therefore
	\[
		\liminf_{n \to \infty } \dfrac{\reg(\Tor_n^A(M,N))}{n} = 2 \quad \mbox{and} \quad \limsup_{n \to \infty } \dfrac{\reg(\Tor_n^A(M,N))}{n} = 3.
	\]
	Furthermore, for any $ \alpha \in (2,3) $, by choosing any subsequence $ n_{\alpha}(l) $ such that $ | n_{\alpha}(l) - \lfloor 2^l/(\alpha - 1) \rfloor | $ is bounded for all $ l \ge 1 $,
	\[
		\lim_{l \to \infty } \dfrac{\reg(\Tor_{n_{\alpha}(l)}^A(M,N))}{n_{\alpha}(l)} = \alpha.
	\]
	In particular, $ n_{\alpha}(l) $ can be a sequence of even (resp. odd) integers.
	Thus both
	\[
		\{\reg(\Tor_{2n}^A(M,N))/(2n) : n \ge 1\} \quad \mbox{and} \quad \{\reg(\Tor_{2n+1}^A(M,N))/(2n+1) : n \ge 1\}
	\]
	are dense sets in $ [2,3] $.
\end{remark}

 Before proving the claims in Example~\ref{exam:reg-Tor-Ext-setup2}, we need to setup some notations and provide some preliminary lemmas.
 
\begin{setup}\label{setup2}
	Along with the hypotheses of Example~\ref{exam:reg-Tor-Ext-setup2}, for every integer $ n \ge 1 $, we set the matrices $ B_n $ and $ C_n $ of order $ n \times (n+1) $ as follows:
	\[ 
	B_n :=  \begin{bmatrix}
	y          & z   &  0        & 0	&  \cdots   & 0 \\
	0          & y   & z       & 0		& \cdots   & 0 \\
	0          & 0   &  y        & z    & \cdots   & 0 \\
	\vdots  & \vdots  & \vdots & \ddots  & \ddots & \vdots \\
	0        &   0   &  0		 &    \cdots	& y     & z \\
	\end{bmatrix}  \mbox{  and  }
	C_n := \begin{bmatrix}
	v          & w   &  0        & 0	&  \cdots   & 0 \\
	0          & v   & w       & 0		& \cdots   & 0 \\
	0          & 0   &  v        & w    & \cdots   & 0 \\
	\vdots  & \vdots  & \vdots & \ddots  & \ddots & \vdots \\
	0        &   0   &  0		 &    \cdots	& v     & w \\
	\end{bmatrix}.
	\]
	Setting $ I_n $ as the $ n \times n $ identity matrix, we construct the block matrices $ E_n $ and $ F_n $ both of order $ {n+1 \choose 2} \times {n+2 \choose 2} $ as follows:
	\[ 
	E_n := 
	\left[\begin{array}{c|c|c|c|c}
	x I_1 & B_1 	&      		&  			&    		\\ \hline
	& x I_2   & B_2	 & 			& 			\\ \hline
	&			& \ddots & \ddots &   		\\ \hline
	& 			& 			&  x I_n	&  B_n 
	\end{array}	\right]  \mbox{ and }
	F_n := \left[\begin{array}{c|c|c|c|c}
	u I_1 & C_1 	&     		&  			&    		\\ \hline
	& u I_2 & C_2		& 			& 			\\ \hline
	&			& \ddots & \ddots &   		\\ \hline
	& 			&			&  u I_n	&  C_n 
	\end{array}	\right]
	\]

	Finally, we set the block matrix
	\[
	D_n := \left[\begin{array}{c|c}
	E_n & \\ \hline
	F_n & E_n
	\end{array}	\right]
	\quad \mbox{of order }\; 2{n+1 \choose 2} \times 2{n+2 \choose 2}.
	\]
	Here the empty blocks in $ E_n $, $ F_n $ and $ D_n $ are filled with zero matrices of suitable order.
\end{setup}

In view of Proposition~\ref{prop:BC+CB=0}, replacing $ v^m $ and $ w^m $ by $ v $ and $ w $ respectively, since $ \Char(K) = 2 $, one obtains the following relations.

\begin{remark}\label{rmk:BC+CB=0}
	With Setup~\ref{setup2}, $ B_n C_{n+1} + C_n B_{n+1} = 0 $ for every $ n \ge 1 $.
\end{remark}

A similar relation holds for $ E_n $ and $ F_n $, which helps us to build minimal free resolution of $ M $.

\begin{proposition}\label{prop:EF+FE=0}
	With Setup~\ref{setup2}, $ E_n F_{n+1} + F_n E_{n+1} = 0 $ for every $ n \ge 1 $.
\end{proposition}

\begin{proof}
	For every $ n \ge 1 $, the block matrix multiplication yields that
	\begin{align*}
	E_n F_{n+1} & = \left[\begin{array}{c|c|c|c|c|c}
	xu I_1 & x C_1 + u B_1	& B_1 C_2	&      		&    		\\ \hline
	& xu I_2 	& x C_2 + u B_2	& B_2 C_3 &  			\\ \hline
	&			& \ddots & \ddots & \ddots 	\\ \hline
	& 			& 			&  xu I_n & x C_n + u B_n & B_n C_{n+1}
	\end{array}	\right], \\
	F_n E_{n+1} & = \left[\begin{array}{c|c|c|c|c|c}
	ux I_1 & u B_1 + x C_1 & C_1 B_2	&      		&    		\\ \hline
	& ux I_2 	& u B_2 + x C_2	& C_2 B_3 & 			\\ \hline
	&			& \ddots & \ddots & \ddots 	\\ \hline
	& 			& 			&  ux I_n & u B_n + x C_n & C_n B_{n+1}
	\end{array}	\right].
	\end{align*}

	Hence `$ \Char(K) = 2 $' and Remark~\ref{rmk:BC+CB=0} yield that $ E_n F_{n+1} + F_n E_{n+1} = 0 $ for every $ n \ge 1 $.
\end{proof}

We compute $ \Tor_n^A(M,N) $ $ (n \ge 1) $ by constructing a graded minimal free resolution of $ M $.

\begin{lemma}\label{lem:reso-setup2}
	With Setup~\ref{setup2}, the following statements hold true.\\
	{\rm (i)} A graded minimal free resolution of $ N $ over $ A $ is given by $ \mathbb{F}_{\bullet}^N : $
	\[
	\cdots \longrightarrow A(-n)^{{n+2 \choose 2}} \xrightarrow{\; E_n \;} A(-n+1)^{{n+1 \choose 2}} \longrightarrow \cdots \xrightarrow{\; E_2 \;} A(-1)^3 \xrightarrow{\; E_1 \;} A \longrightarrow 0
	\]
	{\rm (ii)} A graded minimal free resolution of $ M $ over $ A $ is given by $ \mathbb{F}_{\bullet}^M : $
	\[
	\cdots \xrightarrow{D_{n+1}} \hspace{-0.2cm}
	\begin{array}{c}
	A(-n)^{{n+2 \choose 2}} \\
	\bigoplus \quad \quad \\
	A(-n)^{{n+2 \choose 2}} 
	\end{array}
	\hspace{-0.2cm}
	\xrightarrow{\; D_n \;} \begin{array}{c}
	A(-n+1)^{{n+1 \choose 2}} \\
	\bigoplus \quad \quad \\
	A(-n+1)^{{n+1 \choose 2}} 
	\end{array}
	 \hspace{-0.2cm}
	\hspace{-0.2cm}
	\longrightarrow \cdots
	\xrightarrow{ D_2 } \hspace{-0.2cm}
	\begin{array}{c}
	A(-1)^3\\
	\bigoplus\\
	A(-1)^3
	\end{array}
	\hspace{-0.2cm} \xrightarrow{\; D_1 \;}
\hspace{-0.1cm}
	\begin{array}{c}
	A\\
	\bigoplus\\
	A
	\end{array}
	\hspace{-0.1cm} \longrightarrow 0
	\]
\end{lemma}

\begin{proof}
	The proof is almost same as that of Lemma~\ref{lem:reso-setup1}. So we just mention the steps here.
	
	(i) Set $ N_1 := A/(x) $ and $ N_2 := A/(y,z) $. Then
	\begin{align*}
	& \mathbb{F}_{\bullet}^{N_1} : \;\; \cdots \rightarrow A(-2) \xrightarrow{\; x \;} A(-1) \xrightarrow{\; x \;} A \xrightarrow{\; x \;} 0 \quad \mbox{and} \\
	& \mathbb{F}_{\bullet}^{N_2} : \;\; \cdots \stackrel{B_3}{\longrightarrow}A(-2)^3 \stackrel{B_2}{\longrightarrow} A(-1)^2 \stackrel{B_1}{\longrightarrow} A \rightarrow 0
	\end{align*}
	are graded minimal $ A $-free resolutions of $ N_1 $ and $ N_2 $ respectively, where $ \mathbb{F}_{\bullet}^{N_2} $ is obtained as in Lemma~\ref{lem:reso-setup1}(i). Set $ \mathbb{F}_{\bullet} := \mathbb{F}_{\bullet}^{N_1} \otimes_A \mathbb{F}_{\bullet}^{N_2} $. Hence $ H_i(\mathbb{F}_{\bullet}) = \Tor_i^A(N_1,N_2) = 0 $ for all $ i \ge 1 $ (since $ \mathbb{F}_{\bullet}^{N_1} \otimes_A N_2 $ is acyclic). Therefore $ \mathbb{F}_{\bullet} $ is a free resolution of $ N_1 \otimes_A N_2 = A/(x,y,z) = N $. The assertion follows because $ \mathbb{F}_{\bullet} $ is same as the given free resolution $ \mathbb{F}_{\bullet}^N $.
	
	(ii) Set $ \mathbb{G} := \mathbb{F}_{\bullet}^N $ and $ \mathbb{H} := \mathbb{G}[1] $, i.e., $ \mathbb{H}_n = \mathbb{G}_{n+1} $ and $ d_n^{\mathbb{H}} = (-1) d_{n+1}^{\mathbb{G}} $ for every $ n $. We construct a map $ f : \mathbb{H} \to \mathbb{G} $ as follows: the $ n $th component $ f_n : \mathbb{H}_n \to \mathbb{G}_n $ of $ f $ is defined by $ (-1) F_{n+1} $. By virtue of Proposition~\ref{prop:EF+FE=0}, $ f $ is a homogeneous map of chain complexes. As in the proof of Lemma~\ref{lem:reso-setup1}(ii), the mapping cone of $ f $ provides the desired free resolution $ \mathbb{F}_{\bullet}^M $.
\end{proof}

\begin{para}[Computations of $ \Tor_i^A(M,N) $ and $ \Ext_A^i(M,N) $ with Setup~\ref{setup2}]\label{para:example2-tor-ext}
	In view of Lemma~\ref{lem:reso-setup2}(ii),  by considering the complex $ \mathbb{F}_{\bullet}^M \otimes_A N $ as in \ref{para:example1-tor-ext}, we compute that
	\begin{equation}\label{Tor-n-setup2}
		\Tor_n^A(M,N) = \begin{array}{l}
		\Ker\Big( N(-n)^{{n+2 \choose 2}} \xrightarrow{\; F_n \;}  N(-n+1)^{{n+1 \choose 2}}  \Big) \\
		\;\; \bigoplus \\
		\Coker\Big(  N(-n-1)^{{n+3 \choose 2}} \xrightarrow{\; F_{n+1} \;}  N(-n)^{{n+2 \choose 2}} \Big)
		\end{array} \mbox{ for } n \ge 1.
	\end{equation}
	It follows that $ \Tor_n^A(M,N)_{\mu} = 0 $ for every $ \mu < n $, and $ \Tor_n^A(M,N)_n \neq 0 $. Therefore
	\begin{equation}\label{indeg-tor-setup2}
		\indeg\left( \Tor_n^A(M,N)  \right) = n \quad \mbox{for every } \; n \ge 1.
	\end{equation}
	To compute Ext modules, we consider the complex $ \Hom_A(\mathbb{F}_{\bullet}^M,N) $, which yields that
	\begin{equation}\label{Ext-n-setup2}
		\Ext^n_A(M,N) = \begin{array}{l}
		\Coker\Big( N(n-1)^{{n+1 \choose 2}} \xrightarrow{\;\; F^t_n \;\;} N(n)^{{n+2 \choose 2}} \Big) \\
		\;\; \bigoplus \\
		\; \Ker\Big( N(n)^{{n+2 \choose 2}} \xrightarrow{\;\; F^t_{n+1} \;\;} N(n+1)^{{n+3 \choose 2}} \Big)
		\end{array} \mbox{ for } n \ge 1,
	\end{equation}
	where $ F_n^t $ is the transpose of $ F_n $. It follows from \eqref{Ext-n-setup2} that
	\begin{equation}\label{indeg-ext-setup2}
		 \indeg\left( \Ext_A^n(M,N)  \right) = -n \quad \mbox{for every } \; n \ge 1.
	\end{equation}
\end{para}

 In order to compute regularity of $ \Tor_n^A(M,N) $ and $ \Ext^n_A(M,N) $, we interpret the matrix maps $ F_n $ and $ F_n^t $ in different ways.
 
\begin{definition}
	For a ring $ S $, we denote by $ \delta_X : S[X] \to S[X] $ the $ S $-linear map defined by
	\[
		X^a \longmapsto \left\{ \begin{array}{ll}
			X^{a-1} & \mbox{if } a \ge 1, \\
			0 & \mbox{else}.
		\end{array} \right.
	\]
\end{definition}

\begin{para}[Interpretations of $ F_n $ and $ F_n^t $]\label{para:interpret-Fn}
	Set $ R := K[U,V,W] $, polynomial ring over a field $ K $ of characteristic $ 2 $. Consider the sequences of graded $ R $-linear maps (which are not complexes):
	\[
		\mathbb{F}_{\bullet}^U : \quad \cdots \longrightarrow R(-3) \xrightarrow{\; U \;} R(-2) \xrightarrow{\; U \;} R(-1) \xrightarrow{\; U \;} R \longrightarrow 0,
	\]
	similarly $ \mathbb{F}_{\bullet}^V $ and $ \mathbb{F}_{\bullet}^W $. Set $ \mathbb{F}_{\bullet} := \mathbb{F}_{\bullet}^U \otimes_R \mathbb{F}_{\bullet}^V \otimes_R \mathbb{F}_{\bullet}^W $, which can be defined exactly in the same way as tensor product of complexes is defined. In view of Lemma~\ref{lem:reso-setup2}(i) and its proof, the $ n $th map of the sequence $ \mathbb{F}_{\bullet} $ is given by
	\[
		R(-n)^{\binom{n+2}{2}} \xrightarrow{\; F_n \;} R(-n+1)^{\binom{n+1}{2}},
	\]
	where $ F_n $ is obtained from Setup~\ref{setup2} by replacing $ u, v, w $ with $ U, V, W $ respectively. Identifying the free summand $ R(-n) $ corresponding to $ \mathbb{F}^U_{a_1} \otimes \mathbb{F}^{V}_{a_2} \otimes \mathbb{F}^{W}_{a_3} $ with $ R X^{a_1} Y^{a_2} Z^{a_3} \subseteq R[X,Y,Z]_n $, where $a_1+a_2+a_3=n$ and $a_i \ge 0$, one obtains an $ R $-module isomorphism $ R(-n)^{\binom{n+2}{2}} \xrightarrow{\; \cong \;} R[X,Y,Z]_n $. On the other hand, labeling the basis elements of $ R(-n)^{{n+2 \choose 2}} $ by $ e_{(a_1,a_2,a_3)} $, the action of $F_n$ on $e_{(a_1,a_2,a_3)}$ can be described as follows:
	\[
		F_n \left( e_{(a_1,a_2,a_3)} \right) = \epsilon_1 Ue_{(a_1-1,a_2,a_3)}+\epsilon_2 Ve_{(a_1,a_2-1,a_3)}+\epsilon_3 We_{(a_1,a_2,a_3-1)},
	\]
	where $ \epsilon_i = 1 $ if $ a_i \ge 1 $, else $ \epsilon_i = 0 $.  Hence it can be checked that the diagram
	\begin{equation}\label{Fn-nad-delta-comm-diag}
		\xymatrix{
			R(-n)^{\binom{n+2}{2}} \ar[r]^{F_n \quad} \ar[d]_{\cong} & R(-n+1)^{\binom{n+1}{2}} \ar[d]^{\cong} \\
			R[X,Y,Z]_n \ar[r]^{\delta} & R[X,Y,Z]_{n-1}		}
	\end{equation}
	is commutative, where $ \delta := U \delta_X + V \delta_Y + W \delta_Z $, which is an $ R $-linear map. Dualizing the commutative diagram \eqref{Fn-nad-delta-comm-diag}, or dualizing the above notion, one obtains another commutative diagram
	\begin{equation}\label{Fn-trans-nad-delta-comm-diag}
		\xymatrix{
			R(n)^{\binom{n+2}{2}}  & R(n-1)^{\binom{n+1}{2}} \ar[l]_{F_n^t \quad}  \\
			R[X,Y,Z]_n \ar[u]^{\cong} & R[X,Y,Z]_{n-1}	\ar[l]_{\mu_n}	\ar[u]_{\cong} }
	\end{equation}
	where $ \mu_n $ is an $ R $-linear map defined by multiplication with $ UX+VY+WZ $. Since $ \mu_n $ is injective, it follows that the map given by $ F_n^t $ is an injective map.
\end{para}

The origin of the nonlinear behavior of regularity in Example~\ref{exam:reg-Tor-Ext-setup2}(ii) rely on the behavior of coefficient ideals in positive characteristic.

\begin{lemma}\label{lem:coeff-ideals}
	Set $ R := K[U,V,W] $, where $ \Char(K) = 2 $. For every $ n \ge 1 $, let $ \mathcal{B}_n $ be the set of all monomials in $ U, V, W $ which are the coefficients of $ (UX+VY+WZ)^n $, and $ I_n $ be the ideal of $ R $ generated by $ \mathcal{B}_n $. Then $ \reg(R/I_n) = 3(2^l-1) $ if $ 2^l \le n \le 2^{l+1} - 1 $ for some $ l \ge 0 $.
\end{lemma}

\begin{proof}
	Writing $ n $ in base $ 2 $, $ n = \sum_{i \ge 0} a_i 2^i $ with $ a_i \in \{0,1\} $, set $ l := \max\{ i : a_i \neq 0 \} $ and $ S_n := \{ i : a_i \neq 0 \} $. Since $ \Char(K) = 2 $,
	\begin{align*}
		\sum_{U^a V^b W^c \in \mathcal{B}_n} U^a V^b W^c X^a Y^b Z^c & = (UX+VY+WZ)^n  \\
		& = \prod_{i \in S_n} \left(U^{2^i}X^{2^i}+V^{2^i}Y^{2^i}+W^{2^i}Z^{2^i}\right).
	\end{align*}
	Since $ \sum_{0 \le j \le r} 2^j < 2^{r+1} $ for any $ r \ge 0 $, the above equalities show that the map $ \prod_{i \in S_n} \mathcal{B}_{2^i} \to \mathcal{B}_n $ sending a tuple of monomials to their product is a bijection. Therefore $ I_n = \prod_{a_i = 1} I_{2^i} $. It shows that the minimal number of generators of $ I_n $ is $ 3^{|S_n|} $, a fact that we will not use for the proof.
	
	We now use induction on $ l $. Since $ \reg(R/I_1) = 0 $, for $ l = 0 $, the assertion holds. Suppose $ \reg(R/I_n) = 3(2^l-1) $ if $ 2^l \le n \le 2^{l+1} - 1 $ for some $ l \ge 0 $. Since $ R/I_n $ is Artinian, and the regularity is given by the shifts in the last component of the minimal free resolution $ \mathbb{F}_{\bullet}^{R/I_n} $, applying the Frobenius map, we get $ \mathbb{F}_{\bullet}^{R/I_{2n}} $. So
	\begin{equation}\label{even-reg}
		\reg(R/I_{2n}) = 2(\reg(R/I_n) + 3) - 3 = 3(2^{l+1} - 1) \mbox{ if } 2^{l+1} \le 2n \le 2^{l+2} - 2.
	\end{equation}
	Note that $ I_{2n+1} = \mathfrak{m} I_{2n} $, where $ \mathfrak{m} = (U,V,W) $. Considering the exact sequence
	\[
		0 \longrightarrow I_{2n}/\mathfrak{m} I_{2n} \longrightarrow R/\mathfrak{m} I_{2n} \longrightarrow R/I_{2n} \longrightarrow 0,
	\]
	for every $ 2^{l+1}+1 \le 2n+1 \le 2^{l+2} - 1 $,
	\begin{align}
		\reg(R/I_{2n+1}) & = \max\{ \reg(R/I_{2n}), \reg(I_{2n}/\mathfrak{m} I_{2n}) \} \label{odd-reg}\\
		& = \max\{ 3(2^{l+1} - 1), 2n \} = 3(2^{l+1} - 1). \nonumber
	\end{align}
	Thus the assertion for $ l+1 $ follows from \eqref{even-reg} and \eqref{odd-reg}.
\end{proof}

 Using the interpretation of $ F_n $ given in \ref{para:interpret-Fn}, we now prove the following facts.
 
\begin{lemma}\label{lem:facts-Fn}
	Set $ R := K[U,V,W] $, where $ \Char(K) = 2 $. Then the $ R $-linear map $ \Phi_n : R(-n)^{\binom{n+2}{2}}\xrightarrow{\; F_n \;} R(-n+1) ^{\binom{n+1}{2}} $ has the following properties.
	\begin{enumerate}[{\rm (i)}]
		\item For every $n \ge 1$, $ \Coker(\Phi_n) $ is an Artinian $ R $-module.
		\item For every $n \ge 1$, $\reg (\Coker(\Phi_n)) \le \reg (\Coker(\Phi_{n+1})) - 1 $.
		\item $ \reg (\Coker(\Phi_1 \Phi_2 \cdots \Phi_n)) = 3(2^l - 1) $ if $ 2^l \le n \le 2^{l+1} - 1 $ for some $ l \ge 0 $.
		\item $ \reg (\Coker (\Phi_n))= 2(n-1)$ if $ n=2^l-1 $ for some $ l \ge 1 $.
		\item $ \reg(\Coker(\Phi_n)) = 2(2^l-1) + n - 1 $ if $ 2^l \le n \le 2^{l+1} - 1 $ for some $ l \ge 0 $.
	\end{enumerate}
\end{lemma}

\begin{proof}
	(i) Let $ I(F_n) $ be the ideal of maximal minors of $F_n$.  By construction of $ F_n $, and changing the role of $ U, V $ and $ W $, one can see that $ \big( U^{\binom{n+1}{2}}, V^{\binom{n+1}{2}}, W^{\binom{n+1}{2}} \big) \subset I(F_n) $. Therefore the assertion follows from the fact that $ \Supp(\Coker(\Phi_n)) \subseteq \Supp(R/I(F_n)) $, which is shown in \cite[20.4 and 20.7.a]{Eis95}. 
	
	(ii) By virtue of (i), $ \reg(\Coker(\Phi_n)) $ is the smallest number $ r $ such that $ \Phi_n $ is surjective on the graded components $ \ge r + 1 $. Set $ \Psi_n : R(-1)^{\binom{n+2}{2}} \xrightarrow{\; F_n \;} R^{\binom{n+1}{2}} $, which is same as $ \Phi_n $ but the grading is shifted by $ n - 1 $. So $ \reg(\Coker(\Phi_n)) = \reg(\Coker(\Psi_n)) + n - 1 $. It can be derived from
	\[
		F_{n+1} = \left[\begin{array}{cccc|c}
			&	 		&  	F_n  & 				& 0 \\ \hline
		0	& \cdots & 0 	& U I_{n+1} & C_{n+1}
		\end{array} \right] : {\begin{array}{lll}
			R(-1)^{\binom{n+2}{2}} & \xrightarrow{F_n} & R^{\binom{n+1}{2}} \\
			\;\; \bigoplus & \searrow & \;\bigoplus \\
			R(-1)^{n+2} & \longrightarrow & R^{n+1}
			\end{array}}
	\]
	that $ \reg(\Coker(\Psi_n)) \le \reg(\Coker(\Psi_{n+1})) $, and hence 
	\[
		\reg(\Coker(\Phi_n)) \le \reg(\Coker(\Phi_{n+1})) - 1.
	\]
	
	(iii) In view of the diagram \eqref{Fn-nad-delta-comm-diag}, the composition $ \Phi_1 \Phi_2 \cdots \Phi_n $ can be interpreted by the map $ \delta^n : R[X,Y,Z]_n \longrightarrow R $, where $ \delta = U \delta_X + V \delta_Y + W \delta_Z $. Therefore, by the dual diagram \eqref{Fn-trans-nad-delta-comm-diag}, $ \Image(\Phi_1 \Phi_2 \cdots \Phi_n) $ is equal to the coefficient ideal of $ (U X + V Y + W Z)^n $.
	Hence the result follows from Lemma~\ref{lem:coeff-ideals}.
	
	(iv) Let $ \mathcal{B}'_n := \left\lbrace m_1,\dots, m_{{n+2 \choose 2}} \right\rbrace $ be the set of monomial generators of $ R[X,Y,Z]_n $ ordered by lex with $ X \succ Y \succ Z $. Let $ \mathcal{A}_n $ be the $R$-submodule of $ R(X,Y,Z) $ generated by the ordered set
	\[
		\mathcal{B}''_n := \left\lbrace \dfrac{X^nY^nZ^n}{m_i} : 1 \le i \le {n+2 \choose 2} \right\rbrace.
	\]
	Clearly, $ R[X,Y,Z]_n $ and $ \mathcal{A}_n $ both are free $ R $-modules of same rank with ordered bases $ \mathcal{B}'_n $ and $ \mathcal{B}''_n $ respectively. Consider the $ R $-linear map $ \delta^n : \mathcal{A}_n \rightarrow R[X,Y,Z]_{n} $ defined by acting $ \delta^n $ on the basis elements of $ \mathcal{A}_n $, where $ \delta = U \delta_X + V \delta_Y + W \delta_Z $. Let $ G_n $ be the matrix representation of $ \delta^n $ with respect to the described bases. Thus we have a commutative diagram
	\begin{equation}\label{comm-diag-Gn-Fn}
		\xymatrix{
		R(-2n)^{\binom{n+2}{2}} \ar[r]^{G_n} \ar[d]^{\cong} & R(-n)^{\binom{n+2}{2}} \ar[r]^{F_n} \ar[d]^{\cong} & R(-n+1)^{\binom{n+1}{2}} \ar[d]^{\cong} \\
		\mathcal{A}_n \ar[r]^{\delta^n} & R[X,Y,Z]_n \ar[r]^{\delta} & R[X,Y,Z]_{n-1}.
	}
	\end{equation}
	Since $ n = 2^l - 1 $, the composition $ \delta^{n+1} : \mathcal{A}_n \rightarrow R[X,Y,Z]_{n-1}$ is a zero map. It follows that the top row of \eqref{comm-diag-Gn-Fn} is also a complex.
	Writing $ m_i = X^{a_{i_1}} Y^{a_{i_2}} Z^{a_{i_3}} $ for $ 1 \le i \le {n+2 \choose 2} $, the matrix $ G_n $ can be expressed as
	\begin{equation*}
		(G_n)_{(i,j)} = \epsilon_{(i,j)} U^{n-a_{i_1}-a_{j_1}} V^{n-a_{i_2}-a_{j_2}} W^{n-a_{i_3}-a_{j_3}},
	\end{equation*}
	where $ \epsilon_{(i,j)} = 1 $ if $ U^{n-a_{i_1}-a_{j_1}} V^{n-a_{i_2}-a_{j_2}} W^{n-a_{i_3}-a_{j_3}} \in \mathcal{B}_n $ as defined in Lemma~\ref{lem:coeff-ideals}, and  $ \epsilon_{(i,j)} = 0 $ else. Therefore $ G_n $ is a symmetric matrix. Hence
	\begin{align}
		0 \longrightarrow R(-2n-1)^{\binom{n+1}{2}} & \xrightarrow{\; F_n^t \;} R(-2n)^{\binom{n+2}{2}} \label{res-coker-Fn}\\
		& \xrightarrow{\; G_n \;} R(-n)^{\binom{n+2}{2}} \xrightarrow{\; F_n \;} R(-n+1)^{\binom{n+1}{2}} \longrightarrow 0\nonumber
	\end{align}
	is a complex. Note that the ideal of maximal minors of $ F_n $ has depth $ 3 $. On the other hand, choosing the $ (n+1) $ rows and columns of $ G_n $ indexed by
	\[
		\left\{ X^{n-i} Y^i : 0 \le i \le n \right\} \quad \mbox{and} \quad \left\{ X^n Y^n Z^n/X^{n-j} Y^j : 0 \le j \le n \right\} 
	\]
	respectively, the corresponding submatrix is antidiagonal with entries $ W^n $ on the antidiagonal. Similarly, one may consider suitable minors for $ U $ and $ V $. Thus the ideal $ I_{n+1}(G_n) $ of all $ (n+1) $ minors of $ G_n $ contains pure powers of $ U $, $ V $ and $ W $. So $ \depth(I_{n+1}(G_n),R) = 3 $. Therefore, by Buchsbaum-Eisenbud acyclicity criterion \cite[Thm. 1.4.13]{BH98}, \eqref{res-coker-Fn} is acyclic. So $ \reg(\Coker(\Phi_n)) = 2(n-1) $.
	
	(v) Set $ g(n) := \reg (\Coker(\Phi_1 \Phi_2 \cdots \Phi_n)) $. It follows from (i) that every $ \Phi_n $ is surjective on all high enough graded components. Let $ n = 2^l $. Then, by (iii), $ g(n) > g(n-1) $, which implies that the component $ [\Phi_1 \Phi_2 \cdots \Phi_{n-1}]_{g(n)} $ is onto, but $ [\Phi_1 \Phi_2 \cdots \Phi_n]_{g(n)} $ is not onto. Therefore $ [\Phi_n]_{g(n)} $ is not onto, and hence $ \reg(\Coker(\Phi_n)) \ge g(n) = 3(n-1) $ by (iii). Along with this inequality,  the statements (ii) and (iv) yield that
	\begin{align*}
		3(n - 1) & \le \reg(\Coker(\Phi_n)) \\
		& \le \reg(\Coker(\Phi_{n+1})) - 1 \nonumber \\ 
		& \;\; \vdots  \nonumber \\
		& \le \reg(\Coker(\Phi_{2n - 1})) - (n-1) \nonumber \\
		& =  2(2n - 2) - (n-1) = 3(n-1).\nonumber
	\end{align*}
	Therefore all the above inequalities must be equalities, and it follows that
	\[
		\reg(\Coker(\Phi_m)) = 2(2^l-1) + m - 1 \mbox{ if } 2^l \le m \le 2^{l+1} - 1 \mbox{ for some } l \ge 0.
	\]
\end{proof}

	With all the ingredients in Lemma~\ref{lem:facts-Fn}, we are now able to compute the regularity of Ext and Tor modules in Example~\ref{exam:reg-Tor-Ext-setup2}.

\begin{proof}[Proof of Example~\ref{exam:reg-Tor-Ext-setup2}]
	The expressions for $ \indeg $ are shown in \eqref{indeg-tor-setup2} and \eqref{indeg-ext-setup2}. In view of \eqref{Tor-n-setup2} and \eqref{Ext-n-setup2}, it requires to compute the regularity of  kernel and cokernel of the following maps :
	\begin{align}
		\Phi_n & :=  N(-n)^{{n+2 \choose 2}} \xrightarrow{\;\; F_n \;\;}  N(-n+1)^{{n+1 \choose 2}} \quad \mbox{and} \label{Phi-n-Psi-n-Fn} \\ 
		\Psi_n & := N(n-1)^{{n+1 \choose 2}} \xrightarrow{\;\; F^t_n \;\;} N(n)^{{n+2 \choose 2}} \nonumber
	\end{align}
	for all $ n \ge 1 $. Since $ N $ is annihilated by $ (x,y,z) $, we can substitute $ N $ with $ R:= K[U,V,W] $, and the entries $ u,v,w $ in the matrices $ F_n $ with $ U, V, W $ respectively.
	
	(i) By the observations made in \ref{para:interpret-Fn}, the complex
	\[
		0 \longrightarrow R(n-1)^{{n+1 \choose 2}} \xrightarrow{\; F^t_n \;} R(n)^{{n+2 \choose 2}} \longrightarrow 0
	\]
	is acyclic, and it provides a graded minimal $ R $-free resolution of $ \Coker(\Psi_n) $. Therefore, $ \Ker(\Psi_n) = 0 $ and $ \reg(\Coker(\Psi_n)) = -n $ for every $ n \ge 1 $. Hence the assertion follows from \eqref{Ext-n-setup2}.
	
	(ii) It follows from the Koszul complex of $ U, V, W $ over $ R $ that regularities of $ \Coker(\Phi_1) $ and $ \Ker(\Phi_1) $ are $ 0 $ and $ 2 $ respectively. So we need to focus on $ n \ge 2 $. By virtue of Lemma~\ref{lem:facts-Fn}(v),
	\begin{equation}\label{reg-coker-phi-n}
		\reg(\Coker(\Phi_n)) = 2(2^l-1) + n - 1 \quad \mbox{if } 2^l \le n \le 2^{l+1} - 1.
	\end{equation}
	Thus, for every $ n \ge 2 $, since $ \reg(\Coker(\Phi_n)) > n - 1 $, in view of \eqref{Phi-n-Psi-n-Fn},
	\begin{equation}\label{coker-ker-Phi-n}
		\reg(\Coker(\Phi_n)) = \max\{ n - 1, \reg(\Ker(\Phi_n)) - 2 \} = \reg(\Ker(\Phi_n)) - 2.
	\end{equation}
	Therefore \eqref{reg-coker-phi-n} and \eqref{coker-ker-Phi-n} yield that
	\begin{equation}\label{reg-ker-Phi-n}
		\reg(\Ker(\Phi_n)) = 2(2^l-1) + n + 1 \quad \mbox{if } 2^l \le n \le 2^{l+1} - 1.
	\end{equation}
	It follows from \eqref{Tor-n-setup2}, \eqref{reg-coker-phi-n} and \eqref{reg-ker-Phi-n} that
	\begin{align*}
		& \reg\left( \Tor_n^A(M,N)  \right) = \max\left\{ \reg(\Ker(\Phi_n)), \reg(\Coker(\Phi_{n+1})) \right\} =\\
		& \left\{ \hspace{-0.2cm}\begin{array}{ll}
		\max\left\{ 2(2^l-1) + n + 1,  2(2^l-1) + n \right\} = 2^{l+1} - 1 + n & \mbox{if } 2^l \le n \le 2^{l+1} - 2\\
		\max\left\{ 2(2^l-1) + n + 1,  2(2^{l+1}-1) + n \right\} = 2^{l+2} - 2 + n & \mbox{if } n = 2^{l+1} - 1.
		\end{array}\right. \nonumber
	\end{align*}
	Hence, computing $ \reg(\Tor_1^A(M,N)) = 3 $ separately, the assertion follows.
\end{proof}

\begin{remark}\label{rmk:never-fg-tor}
	Note that by \eqref{Tor-n-setup2} and Lemma~\ref{lem:facts-Fn}(i),
	\[
		H_{A_+}^0 \hspace{-0.15cm}\left( \Tor_{\star}^A(M,N)\right)\hspace{-0.1cm}^{\vee} = \bigoplus_{n \ge 1} (\Coker(\Phi_n))^{\vee}.
	\]
	Hence Lemma~\ref{lem:facts-Fn}(i) and (v) yield that
	\[
		\indeg\left( H_{A_+}^0 \hspace{-0.15cm}\left( \Tor_{n}^A(M,N)\right)\hspace{-0.1cm}^{\vee} \right) = - 2(2^l-1) - n + 1 \mbox{ if } 2^l \le n \le 2^{l+1} - 1.
	\]
	Therefore, by Proposition~\ref{prop:BCH}, one cannot make $ H_{A_+}^0 \hspace{-0.15cm}\left( \Tor_{\star}^A(M,N)\right)\hspace{-0.1cm}^{\vee} $ 
	a finitely generated module over any Noetherian $ \mathbb{Z} $-graded algebra $ A[z_1,\ldots,z_r] $.
\end{remark}

\section*{Acknowledgments}
	
	This work was done during the three months postdoctoral visit of the second named author. He would like to thank LIA Indo-French CNRS Program in Mathematics for their financial support. Computations with Macaulay2 \cite{M2} helped us to find Examples~\ref{exam:reg-Tor-Ext-setup1} and \ref{exam:reg-Tor-Ext-setup2}.
	
\bibliographystyle{acm} 
\bibliography{CGN-v22-final}

\end{document}